\documentclass{mcom-l}

\usepackage{graphicx}
\usepackage{epstopdf}
\usepackage{amsmath}
\usepackage{latexsym}
\usepackage{amsfonts}
\usepackage{multirow}
\usepackage{cases}
\usepackage{tabularx}
\usepackage{booktabs}
\usepackage{color}
\usepackage{setspace}
\usepackage{algorithm}
\usepackage{algorithmic}

\makeatletter
\newenvironment{breakablealgorithm}
  {
   \begin{center}
     \refstepcounter{algorithm}
     \hrule height.8pt depth0pt \kern2pt
     \renewcommand{\caption}[2][\relax]{
       {\raggedright\textbf{\ALG@name~\thealgorithm} ##2\par}%
       \ifx\relax##1\relax 
         \addcontentsline{loa}{algorithm}{\protect\numberline{\thealgorithm}##2}%
       \else 
         \addcontentsline{loa}{algorithm}{\protect\numberline{\thealgorithm}##1}%
       \fi
       \kern2pt\hrule\kern2pt
     }
  }{
     \kern2pt\hrule\relax
   \end{center}
  }
\makeatother

\newtheorem{theorem}{Theorem}[section]
\newtheorem{lemma}{Lemma}[section]
\newtheorem{proposition}{Proposition}[section]

\theoremstyle{definition}
\newtheorem{definition}{Definition}[section]

\newtheorem{assumption}{Assumption}[section]
\newtheorem{corollary}[theorem]{Corollary}

\theoremstyle{remark}
\newtheorem{remark}[theorem]{Remark}
\definecolor{myblue}{rgb}{0,0,.5}

\numberwithin{equation}{section}

\begin{document}

\title[Generalized Jacobi Eigenvalue Algorithm]{Generalized Jacobi Method for Computing Eigenvalues of Dual Quaternion Hermitian Matrices}


\author{Yongjun Chen}
\address{Department of Mathematical Sciences, Tsinghua University, Beijing, 100084, China}
\curraddr{}
\email{chen-yj19@mails.tsinghua.edu.cn}
\thanks{}

\author{Liping Zhang}
\address{Department of Mathematical Sciences, Tsinghua University, Beijing, 100084, China}
\curraddr{}
\email{lipingzhang@mail.tsinghua.edu.cn}
\thanks{Corresponding author: Liping Zhang ({\it E-mail: lipingzhang@tsinghua.edu.cn})}
\thanks{This work was supported by the National Nature Science
Foundation of China (Grant No. 12171271).}


\date{}

\dedicatory{}

\begin{abstract}
Dual quaternion matrices have various applications in robotic research and its spectral theory has been extensively studied in recent years. In this paper, we extend Jacobi method to compute all eigenpairs of dual quaternion Hermitian matrices and establish its convergence. The improved version with elimination strategy is proposed to reduce the computational time. Especially, we present a novel three-step Jacobi method to compute such eigenvalues which have identical standard parts but different dual parts. We prove that the proposed three-step Jacobi method terminates after at most finite iterations and can provide $\epsilon$-approximation of eigenvalue. To the best of our knowledge, both the power method and the Rayleigh quotient iteration method can not handle such eigenvalue problem in this scenario. Numerical experiments illustrate the proposed Jacobi-type algorithms are effective and stable, and also outperform the power method and the Rayleigh quotient iteration method.
\end{abstract}

\maketitle

\section{Introduction}
\label{intro}
Dual quaternion numbers and dual quaternion matrices have various applications in robotic research, including the hand-eye calibration problem \cite{qc7,A6,A10}, the simultaneous localization and mapping (SLAM) problem \cite{qc1,qc2,qc3,qc6,qc15,qc16,A7}, and multi-agent formation control \cite{qc13}. The spectral theory of dual quaternion matrices has been explored in \cite{qc11,A2}, from which we know that when a dual number serves as a right eigenvalue of a square dual quaternion matrix, it also functions as a left eigenvalue of that same matrix, and is consequently referred to as a dual number eigenvalue. Notably, it has been demonstrated that an $n \times n$ dual quaternion Hermitian matrix possesses exactly $n$ eigenvalues for dual numbers.
Subsequently, paper \cite{qc10} delved into the minimax principle pertaining to the eigenvalues of dual quaternion Hermitian matrices.  Hoffman-Wielandt type inequality was established \cite{qc9} for dual quaternion Hermitian matrices by employing von Neumann type trace inequality. Additionally, it was highlighted in \cite{qc13} the importance of the eigenvalue theory of dual quaternion Hermitian matrices in studying the stability of multi-agent formation control systems.

In recent research by Qi and Cui \cite{qc4}, a power method was proposed to compute the eigenvalues of dual quaternion Hermitian matrices. An eigenvalue $\hat{\lambda}$ of a dual quaternion Hermitian matrix is referred to as a strictly dominant eigenvalue with multiplicity $k$ if the matrix possesses $k$ identical eigenvalues $\hat{\lambda}$, and the standard part of $\hat{\lambda}$ exceeds the standard part of the remaining $n-k$ eigenvalues. The linear convergence of the power method was demonstrated when computing the strict dominant eigenvalue and its corresponding eigenvector. However, it was observed that the power method fails to converge when a dual quaternion Hermitian matrix has two eigenvalues with identical standard parts but different dual parts. Moreover, when the dimension of the dual quaternion Hermitian matrix is increasing, the computational efficiency is significantly affected. Subsequently, the Rayleigh quotient iteration method was proposed \cite{qc17} for computing the extreme eigenvalues of dual quaternion Hermitian matrices. An eigenvalue $\hat{\lambda}$ of a dual quaternion Hermitian matrix is referred to as an extreme eigenvalue if the magnitude of its standard part is strictly greater or smaller than that of all other eigenvalues. It was proven that the Rayleigh quotient iteration method exhibits locally cubic convergence, which is substantially faster than the power method. The eigenvalue to which the Rayleigh quotient iteration method converges depends on the initial iteration point. Therefore, selecting an appropriate initial point is important. Additionally, at each iteration, the Rayleigh quotient iteration method requires solving a linear equation. This can be time-consuming for large dimensions. Similar to the power method, the Rayleigh quotient iteration method fails to converge when a dual quaternion Hermitian matrix has two eigenvalues with identical standard parts but different dual parts.

Motivated by these challenges, our objective is to develop an efficient and stable method for calculating all the eigenvalues of dual quaternion Hermitian matrices. The Jacobi eigenvalue algorithm is an old but stable iterative method to find the eigenvalues and eigenvectors of a real-symmetric matrix and is suitable for parallel calculations. In this paper, we extend the applicability of the Jacobi eigenvalue algorithm to the realm of dual quaternion Hermitian matrices. The Jacobi eigenvalue algorithm exhibits favorable numerical performance in terms of accuracy and stability, making it a suitable candidate for such a generalization. The key way of our extension is to  derive a specific formula for a dual quaternion unitary matrix, which can be used to diagonalize a $2\times 2$ dual quaternion Hermitian matrix. Subsequently, based on this formula, we extend the Jacobi eigenvalue algorithm to cater to the characteristics of dual quaternion Hermitian matrices. Very recently,  Jacobi method was also extended by Ding et al. \cite{qc123} to compute eigenvalues of dual quaternion Hermitian matrices,  but its strategy to diagonalize a $2\times 2$ dual quaternion Hermitian matrices differs from our Givens matrix defined in Definition \ref{J}, and the acceleration method employed to avoid searching for maximum element also varies. In addition, there was without convergence analysis in \cite{qc123}.

The structure of this paper is organized as follows. Section \ref{Preliminaries} provides an overview of the fundamental definitions and results pertaining to dual quaternions and dual quaternion matrices. In Section \ref{Jacobi_Method}, we delve into the diagonalization of a $2 \times 2$ dual quaternion Hermitian matrices and propose three Jacobi-type algorithms for computing the eigenvalues and eigenvectors of dual quaternion Hermitian matrices. Moreover, convergence analyses for these algorithms are also established. Section \ref{Experiments} presents numerical results to show the efficiency and accuracy of the generalized Jacobi-type eigenvalue algorithms. Some final remarks are drawn in Section \ref{summery}.

\section{Preliminaries}\label{Preliminaries}

In this section, we review some preliminary knowledge of dual numbers, quaternion, dual quaternion, dual quaternion matrices and eigenvalues of dual quaternion matrices.

\subsection{Dual quaternions}

Let $\mathbb{R}$, $\mathbb{D}$, $\mathbb{DC}$, $\mathbb{Q}$, $\mathbb{U}$, $\hat{\mathbb{Q}}$, and $\hat{\mathbb{U}}$  respectively denote the sets of real numbers, dual numbers, dual complex numbers, quaternion, unit quaternion, dual quaternion, and unit dual quaternion.

Let the symbol $\varepsilon$ denote the infinitesimal unit which satisfies $\varepsilon\neq 0,\varepsilon^{2}=0$. $\varepsilon$ is commutative with complex numbers and quaternions.

\begin{definition}\label{def1}
A {\bf dual complex number} $\hat{a}=a_{st}+a_{\mathcal I}\varepsilon\in \mathbb{DC}$ has standard part
$a_{st}\in \mathbb{C}$ and dual part $a_{\mathcal I}\in \mathbb{C}$. We say that $\hat{a}$ is {\bf appreciable} if $a_{st}\neq 0$. If $a_{st},a_{\mathcal I}\in \mathbb{R}$, then $\hat{a}$ is called a {\bf dual number}.
\end{definition}

The following definition lists some operators about dual complex numbers, which can be found in \cite{qc12}.

\begin{definition}\label{def2}
Let $\hat{a}=a_{st}+a_{\mathcal I}\varepsilon$ and $\hat{b}=b_{st}+b_{\mathcal I}\varepsilon$ be any two {\bf dual complex numbers}. The conjugate,  absolute value of $\hat{a}$, and the addition, multiplication, division between $\hat{a}$ and $\hat{b}$ are defined as follows.
\begin{itemize}
\item[{\rm (i)}] The {\bf conjugate} of $\hat{a}$ is
$\hat{a}^{\ast }=a_{st}^{\ast}+a_{\mathcal I}^{\ast }\varepsilon,$ where $a_{st}^{\ast}$ and $a_{\mathcal I}^{\ast}$ are conjugates of the complex numbers $a_{st}$ and $a_{\mathcal I}$, respectively.

\item[{\rm (ii)}] The {\bf addition} and {\bf multiplication} of $\hat{a}$ and $\hat{b}$ are
$$ \hat{a}+\hat{b}=\hat{b}+\hat{a}=\left (a_{st}+b_{st} \right )+\left (a_{\mathcal I}+b_{\mathcal I} \right )\varepsilon,$$
and
$$\hat{a}\hat{b}=\hat{b}\hat{a}=a_{st}b_{st} +\left (a_{st}b_{\mathcal I}+a_{\mathcal I}b_{st} \right )\varepsilon.$$
Clearly, the addition and multiplication operators of dual complex numbers are all commutative.

\item[{\rm (iii)}] When $b_{st}\ne 0$ or $a_{st}=b_{st}=0$, we can define the division operation of dual numbers as
$$\frac{\hat{a}}{\hat{b}}=\begin{cases}
\dfrac{ a_{st}}{b_{st}}+\left (\dfrac{ a_{\mathcal I}}{b_{st}}-\dfrac{ a_{st}b_{\mathcal I}}{b_{st}b_{st}} \right )\varepsilon, & {\rm if} \quad b_{st}\neq 0, \\
\dfrac{ a_{\mathcal I}}{b_{\mathcal I}}+c\varepsilon, & {\rm if} \quad a_{st}=b_{st}=0,
\end{cases}
$$
where $c$ is an arbitrary complex number.

\item[{\rm (iv)}] The {\bf absolute value}  of $\hat{a}$ is
$$\left | \hat{a}\right |=\begin{cases}
\left | a_{st}\right |+{\rm sgn}( a_{st}) a_{\mathcal I}\varepsilon, & {\rm if} \quad a_{st}\neq 0, \\
\left | a_{\mathcal I} \right |\varepsilon, & \text{\rm otherwise.}
\end{cases}$$

\end{itemize}
\end{definition}

The order operation for dual numbers was defined in \cite{qc4} and the limit operation to dual number sequences was defined in \cite{qc12}.

\begin{definition}\label{def3}
Let $\hat{a}=a_{st}+a_{\mathcal I}\varepsilon$, $\hat{b}=b_{st}+b_{\mathcal I}\varepsilon$ be any two {\bf dual numbers} and $\left \{
\hat{a}_{k}=a_{k,st}+a_{k,\mathcal I}\varepsilon : k=1,2,\cdots  \right \}$ be a dual number sequence. We give the order operation between $\hat{a}$ and $\hat{b}$, and extend limit operation to dual number sequences.

\begin{itemize}
\item[{\rm (i)}] We say that $\hat{a} > \hat{b}$ if
$$\text{$a_{st} > b_{st}$ ~~~{\rm or}~~~ $a_{st}=b_{st}$ {\rm and} $a_{\mathcal I} > b_{\mathcal I}$}.$$

\item[{\rm (ii)}] We say that $\left \{ \hat{a}_{k}=a_{k,st}+a_{k,\mathcal I}\varepsilon \right \}$ converges to a limit $\hat{a}=a_{st}+a_{\mathcal I}\varepsilon$ if
$$\underset{k\rightarrow \infty }{\lim} a_{k,st}=a_{st}\quad {\rm and}\quad \underset{k\rightarrow \infty }{\lim} a_{k,\mathcal I}=a_{\mathcal I}.$$

\end{itemize}
\end{definition}

\begin{definition}\label{def4}
A {\bf quaternion} $\tilde{q}=\left [q_{0}, q_{1},q_{2},q_{3}\right ]$ is a real four-dimensional vector. We can rewrite $\tilde{q}=[ q_{0}, \vec{q} ]$, where $\vec{q}$ is a real three-dimensional vector.
\end{definition}

The following definition lists some operators about quaternions; see e.g., \cite{qc5,qc7,qc14}.

\begin{definition}\label{def5}
Let $\tilde{p}=\left [p_{0}, \vec{p}\right ]$ and $\tilde{q}=\left [q_{0}, \vec{q}\right ]$  be any two {\bf quaternions}. The conjugate, magnitude, and inverse of $\tilde{p}$, and the addition, multiplication between $\tilde{q}$ and $\tilde{p}$ are defined as follows.
\begin{itemize}

\item[{\rm (i)}] The {\bf conjugate} of $\tilde{p}$ is $\tilde{p}^{\ast }=\left [p_{0}, -\vec{p}\right ]$.
\item[{\rm (ii)}] The {\bf magnitude} of $\tilde{p}=\left [p_{0}, p_{1},p_{2},p_{3}\right ]$ is defined by $$\left |\tilde{p} \right |=\sqrt{p_{0}^{2}+p_{1}^{2}+p_{2}^{2}+p_{3}^{2}}.$$
\item[{\rm (iii)}] The {\bf addition} and {\bf multiplication} of $\tilde{q}$ and $\tilde{p}$ are
$$\tilde{p}+\tilde{q}=\tilde{q}+\tilde{p}=\left [p_{0}+q_{0}, \vec{p}+\vec{q}\right ]$$
and
$$\tilde{p}\tilde{q}=\left [p_{0}q_{0}-\vec{p}\cdot \vec{q}, p_{0}\vec{q}+q_{0}\vec{p}+\vec{p}\times \vec{q}\right ].$$
Clearly, in general $\tilde{p}\tilde{q}\neq \tilde{q}\tilde{p}$, and $\tilde{p}\tilde{q}=\tilde{q}\tilde{p}$ if and only if $\vec{p}\times\vec{q}=\vec{0}$, i.e., either $\vec{p}=\vec{0}$ or $\vec{q}=\vec{0}$ or $\vec{p}=c\vec{q}$ for some real number $c$.
\item[{\rm (iv)}] Let $\tilde{1}=\left [1, 0,0,0\right ]\in \mathbb{Q}$ be the identity element of $\mathbb{Q}$. If $\tilde{q}\tilde{p}=\tilde{p}\tilde{q}=\tilde{1}$, then $\tilde{p}$ is invertible and $\tilde{p}^{-1}=\tilde{q}$ is called the inverse of $\tilde{p}$.
\end{itemize}
\end{definition}

\begin{definition}\label{def6}
We say that $\tilde{p}\in \mathbb{Q}$ is a {\bf unit quaternion} if $\left |\tilde{p} \right |=1$.
Clearly, if $\tilde{p}$ and $\tilde{q}$ are unit quaternions, i.e., $\tilde{p},\tilde{q}\in \mathbb{U}$, then $\tilde{p}\tilde{q}\in \mathbb{U}$. Furthermore, we have $\tilde{p}^{\ast }\tilde{p}=\tilde{p}\tilde{p}^{\ast }=\tilde{1}$, i.e., $\tilde{p}$ is invertible and $\tilde{p}^{-1}=\tilde{p}^{\ast }$.
\end{definition}

\begin{definition}\label{def7}
A {\bf dual quaternion} $\hat{p}=\tilde{p}_{st}+\tilde{p}_{\mathcal I}\varepsilon\in \hat{\mathbb{Q}}$ has standard part
$\tilde{p}_{st}\in \mathbb{Q}$ and dual part $\tilde{p}_{\mathcal I}\in \mathbb{Q}$. Let $\tilde{0}=[0,0,0,0]\in \mathbb Q$ be the zero element of $\mathbb Q$. We say that $\hat{p}$ is {\bf appreciable} if $\tilde{p}_{st}\neq \tilde{0}$.
\end{definition}

The following definition lists some operators about dual quaternions.

\begin{definition}\label{def8}
Let $\hat{p}=\tilde{p}_{st}+\tilde{p}_{\mathcal I}\varepsilon$ and $\hat{q}=\tilde{q}_{st}+\tilde{q}_{\mathcal I}\varepsilon$ be any two {\bf dual quaternions}. The conjugate, magnitude, inverse of $\hat{p}$, and the addition, multiplication, division between $\hat{p}$ and $\hat{q}$ are defined as follows.
\begin{itemize}
\item[{\rm (i)}] The {\bf conjugate} of $\hat{p}$ is
$\hat{p}^{\ast }=\tilde{p}_{st}^{\ast}+\tilde{p}_{\mathcal I}^{\ast }\varepsilon,$ where $\tilde{p}_{st}^{\ast}$ and $\tilde{p}_{\mathcal I}^{\ast}$ are conjugates of $\tilde{p}_{st},\tilde{p}_{\mathcal I}$ respectively.

\item[{\rm (ii)}] The {\bf addition} and {\bf multiplication} of $\hat{p}$ and $\hat{q}$ are
$$ \hat{p}+\hat{q}=\hat{q}+\hat{p}=\left (\tilde{p}_{st}+\tilde{q}_{st} \right )+\left (\tilde{p}_{\mathcal I}+\tilde{q}_{\mathcal I} \right )\varepsilon,$$
and
$$\hat{p}\hat{q}=\tilde{p}_{st}\tilde{q}_{st} +\left (\tilde{p}_{st}\tilde{q}_{\mathcal I}+\tilde{p}_{\mathcal I}\tilde{q}_{st} \right )\varepsilon.$$
Clearly, the addition operator of dual quaternions is commutative, while the multiplication
 operator is not in general.
\item[{\rm (iii)}] Let $\hat{1}=\tilde{1}+\tilde{0}\varepsilon\in \hat{\mathbb Q}$ be the identity element of $\hat{\mathbb Q}$. If $\hat{q}\hat{p}=\hat{p}\hat{q}=\hat{1}$, then $\hat{p}$ is invertible and $\hat{p}^{-1}=\hat{q}$ is called the inverse of $\hat{p}$.

\item[{\rm (iv)}] The {\bf magnitude} \cite{qc12} of $\hat{p}$ is
$$\left | \hat{p}\right |=\begin{cases}
\left | \tilde{p}_{st}\right |+\frac{{\rm sc}\left (\tilde{p}_{st}^{\ast }\tilde{p}_{\mathcal I} \right )}{\left |  \tilde{p}_{st} \right |}\varepsilon, & {\rm if} \quad \tilde{p}_{st}\neq 0, \\
\left | \tilde{p}_{\mathcal I} \right |\varepsilon, & \text{\rm otherwise,}
\end{cases}$$
where ${\rm sc}\left ( \tilde{p}\right )=\dfrac{1}{2}\left (\tilde{p}+\tilde{p}^{\ast } \right )$ is the scalar part of $\tilde{p}$.

\item[{\rm (v)}] When $\tilde{q}_{st}\ne \tilde{0}$ or $\tilde{q}_{st}=\tilde{q}_{st}=0$, we can define the division operation of dual numbers as
$$\frac{\hat{p}}{\hat{q}}=\begin{cases}
\dfrac{ \tilde{p}_{st}}{\tilde{q}_{st}}+\left (\dfrac{ \tilde{p}_{\mathcal I}}{\tilde{q}_{st}}-\dfrac{ \tilde{p}_{st}\tilde{q}_{\mathcal I}}{\tilde{q}_{st}\tilde{q}_{st}} \right )\varepsilon, & {\rm if} \quad \tilde{q}_{st}\neq \tilde{0}, \\
\dfrac{ \tilde{p}_{\mathcal I}}{\tilde{q}_{\mathcal I}}+\tilde{c}\varepsilon, & {\rm if} \quad \tilde{p}_{st}=\tilde{q}_{st}=0,
\end{cases}
$$
where $\tilde{c}$ is an arbitrary quaternion number.
\end{itemize}
\end{definition}

\begin{definition}\label{def6.1}
We say that $\hat{p}\in \hat{\mathbb{Q}}$ is a {\bf unit dual quaternion} if $\left |\hat{p} \right |=1$. Clearly, if $\hat{p}$ and $\hat{q}$ are unit dual quaternions, i.e., $\hat{p},\hat{q}\in \hat{\mathbb{U}}$, then $\hat{p}\hat{q}\in \hat{\mathbb{U}}$. Furthermore, we have $\hat{p}^{\ast }\hat{p}=\hat{p}\hat{p}^{\ast }=\hat{1}$, i.e., $\hat{p}$ is invertible and $\hat{p}^{-1}=\hat{p}^{\ast }$.
\end{definition}

\subsection{Dual quaternion matrices}

The set of $n\times m$ real, quaternion, and dual quaternion matrix are denoted as $\mathbb{R}^{n\times m}$, $\mathbb{Q}^{n\times m}$, and $\hat{\mathbb{Q}}^{n\times m}$, respectively.
Let $\tilde{\mathbf{O}}^{n\times m}$ and $\hat{\mathbf{O}}^{n\times m}$ denote the $n\times m$ zero quaternion and zero dual quaternion matrices, respectively. Let $\tilde{\mathbf{I}}_n$ and $\hat{\mathbf{I}}_n$ denote the identity quaternion and identity dual quaternion matrices with dimension $n$, respectively. The set of unitary quaternion matrix and unitary dual quaternion matrix with dimension $n$ is denoted as $\mathbb{U}^{n}_2$ and $\hat{\mathbb{U}}^{n}_2$, respectively.
A dual quaternion matrix $\hat{\mathbf Q}=\tilde{\mathbf Q}_{st} +\tilde{\mathbf Q}_{\mathcal I}\varepsilon \in \hat{\mathbb{Q}}^{n\times m}$ is called appreciable, if $\tilde{\mathbf Q}_{st}\neq \tilde{\mathbf O }$. A quaternion matrix $\tilde{\mathbf{U}}\in \mathbb{Q}^{n\times n}$ is a unitary matrix if $\tilde{\mathbf{U}}^\ast\tilde{\mathbf{U}}=\tilde{\mathbf{U}}\tilde{\mathbf{U}}^\ast=\tilde{\mathbf{I}}_{n}$.
A dual quaternion matrix $\hat{\mathbf{U}}\in \hat{\mathbb{Q}}^{n\times n}$ is a unitary matrix if $\hat{\mathbf{U}}^\ast\hat{\mathbf{U}}=\hat{\mathbf{U}}\hat{\mathbf{U}}^\ast=\hat{\mathbf{I}}_{n}$.


The following definition gives the magnitude of  quaternion vectors, $F$-norm and conjugate
transpose of quaternion matrices.
\begin{definition}
Let $\tilde{\mathbf{x}}=\left ( \tilde{x}_{i} \right ) \in \mathbb{Q}^{n\times 1}$ and $\tilde{\mathbf{Q}}=\left ( \tilde{q}_{ij} \right ) \in \mathbb{Q}^{n\times m}$.
\begin{itemize}
\item[{\rm (i)}] The {\bf magnitude} of $\tilde{\mathbf{x}}$ is $$\| \tilde{\mathbf{x}}\|=\sqrt{\sum_{i=1}^{n}|\tilde{x}_{i}|^2}.$$
\item[{\rm (ii)}]
The $F$-norm of $\tilde{\mathbf{Q}}$ is $$\| \tilde{\mathbf{Q}} \|_F=\sqrt{\sum_{ij}^{} |\tilde{q}_{ij}|^2}.$$
\item[{\rm (iii)}]
The {\bf conjugate transpose} of $\tilde{\mathbf{Q}}$ is
$$\tilde{\mathbf{Q}}^{\ast}=( \tilde{q}_{ji}^{\ast }).$$
If $\tilde{\mathbf{Q}}^{\ast}=\tilde{\mathbf{Q}}$, then $\tilde{\mathbf{Q}}$ is called a quaternion Hermitian matrix. The set of quaternion Hermitian matrix with dimension $n$ is denoted as $\mathbb{H}^{n}$.
\end{itemize}
\end{definition}

The following definition lists some operators about dual quaternion matrices.

\begin{definition}\label{def9}
Let $\hat{\mathbf{Q}}=\left ( \hat{q}_{ij} \right ) \in \hat{\mathbb{Q}}^{n\times m}$. The {\bf transpose} and the {\bf conjugate transpose} of $\hat{\mathbf{Q}}$ is defined by
$$\hat{\mathbf{Q}}^{T}=\left( \hat{q}_{ji} \right )\quad {\rm and}\quad \hat{\mathbf{Q}}^{\ast}=\left( \hat{q}_{ji}^{\ast } \right ).$$
If $\hat{\mathbf{Q}}^{\ast}=\hat{\mathbf{Q}}$, then $\hat{\mathbf{Q}}$ is called a dual quaternion Hermitian matrix. The set of dual quaternion Hermitian matrix with dimension $n$ is denoted as $\hat{\mathbb{H}}^{n}$.
\end{definition}

The following definition gives $2$-norm, $2^{R}$-norm for dual quaternion vectors and $F$-norm, $F^{R}$-norm for dual quaternion matrices. See \cite{qc4,qc10,qc12}.

\begin{definition}\label{def10}
Let $\hat{\mathbf{x}}=\left ( \hat{x}_{i} \right ) \in \hat{\mathbb{Q}}^{n\times 1}$ and $\hat{\mathbf{Q}}=\left ( \hat{q}_{ij} \right ) \in \hat{\mathbb{Q}}^{n\times m}$.

The $2$-norm and $2^{R}$-norm for dual quaternion vectors are respectively defined by
\begin{equation}\label{2norm}
\left \| \hat{\mathbf x }\right \|_{2}=\begin{cases}
\sqrt{\sum_{i=1}^{n}\left |\hat{x}_{i} \right |^{2}}, & {\rm if} \quad \tilde{\mathbf x }_{st}\neq \tilde{\mathbf O }, \\
\|\tilde{\mathbf{x}}_{\mathcal{I}}\|\varepsilon, & {\rm if} \quad \tilde{\mathbf x }_{st}=\tilde{\mathbf O },
\end{cases}
\end{equation}
and
\begin{equation}\label{2rnorm}
\left \| \hat{\mathbf x }\right \|_{2^{R}}=\sqrt{\left \| \tilde{\mathbf x}_{st} \right \|^{2}+\left \| \tilde{\mathbf x }_{\mathcal I}\right \|^{2}}.
\end{equation}
The set of $n\times 1$ dual quaternion vectors with unit $2$-norm is denoted as $\hat{\mathbb{Q}}_2^{n\times 1}$.

The $F$-norm and $F^{R}$-norm for dual quaternion matrices are defined by
\begin{equation}\label{fnorm}
\| \hat{\mathbf Q } \|_{F}=\begin{cases}
 \| \tilde{\mathbf Q }_{st} \|_{F}+\frac{{\rm sc}\left ({\rm tr}\left (\tilde{\mathbf Q}_{st}^{\ast }\tilde{\mathbf Q}_{\mathcal I} \right ) \right )}{\|  \tilde{\mathbf Q}_{st} \|_{F}}\varepsilon, & {\rm if} \quad \tilde{\mathbf Q}_{st}\neq \tilde{\mathbf O }, \\
\| \tilde{\mathbf Q }_{\mathcal I} \|_{F}\varepsilon, & {\rm if} \quad \tilde{\mathbf Q }_{st}=\tilde{\mathbf O },
\end{cases}
\end{equation}
and
\begin{equation}\label{frnorm}
 \| \hat{\mathbf Q } \|_{F^{R}}=\sqrt{ \| \tilde{\mathbf Q }_{st} \|_{F}^{2}+ \| \tilde{\mathbf Q }_{\mathcal I} \|_{F}^{2}},
\end{equation}
respectively.
\end{definition}

The following definition introduces eigenvalues and eigenvectors of dual quaternion matrices \cite{qc11}.
\begin{definition}\label{def12}
Suppose that $\hat{\mathbf Q}\in \hat{\mathbb{Q}}^{n\times n}$.

If there exist $\hat{\lambda} \in \hat{\mathbb{Q}}$ and $\hat{\mathbf x}\in \hat{\mathbb{Q}}^{n\times 1}$, where $\hat{\mathbf x}$ is appreciable, such that
\begin{equation}
    \hat{\mathbf Q}\hat{\mathbf x}=\hat{\mathbf x}\hat{\lambda},
\end{equation}
then we call $\hat{\lambda}$ is a {\bf right eigenvalue} of $\hat{\mathbf Q}$ with $\hat{\mathbf x}$ as an associated {\bf right eigenvector}.

If there exist $\hat{\lambda} \in \hat{\mathbb{Q}}$ and $\hat{\mathbf x}\in \hat{\mathbb{Q}}^{n\times 1}$, where $\hat{\mathbf x}$ is appreciable, such that
\begin{equation}
    \hat{\mathbf Q}\hat{\mathbf x}=\hat{\lambda}\hat{\mathbf x},
\end{equation}
then we call $\hat{\lambda}$ is a {\bf left eigenvalue} of $\hat{\mathbf Q}$ with $\hat{\mathbf x}$ as an associated {\bf left eigenvector}.

Since a dual number is commutative with a dual quaternion vector, then if $\hat{\lambda}$ is a dual number and a right eigenvalue of $\hat{\mathbb{Q}}$, it is also a left eigenvalue of $\hat{\mathbb{Q}}$.  In this case, we simply call $\hat{\lambda}$ an  {\bf eigenvalue} of $\hat{\mathbb{Q}}$ with $\hat{\mathbf{x}}$ as an associated {\bf eigenvector}.
\end{definition}

An $n\times n$ dual quaternion Hermitian matrix has exactly $n$ eigenvalues, which are all dual numbers. Similar to the case of Hermitian matrix, we have unitary decomposition of a dual quaternion Hermitian matrix $\hat{\mathbf{Q}}$, namely, there exist a unitary dual quaternion matrix $\hat{\mathbf{U}}\in \hat{\mathbb{U}}^{n}_2$ and a diagonal dual number matrix $\hat{\Sigma} \in \mathbb D^{n\times n}$ such that  $\hat{\mathbf{Q}}=\hat{\mathbf{U}}^*\hat{\Sigma}\hat{\mathbf{U}}$ \cite{qc11}.

The following proposition shows the Hoffman-Wielandt type inequality still holds for dual quaternion Hermitian matrices; see e.g., \cite{qc9}.

\begin{proposition}\label{lemma0}
If $ \hat{\mathbf Q}_1,\hat{\mathbf Q}_2\in \hat{\mathbb{H}}^{n}$, then we have
\begin{equation}\label{H-W}
\|\lambda(\hat{\mathbf Q}_1)-\lambda(\hat{\mathbf Q}_2)\|_2\leq\|\hat{\mathbf Q}_1-\hat{\mathbf Q}_2\|_F,
\end{equation}
where $\lambda(\hat{\mathbf Q}_1)=(\lambda_1(\hat{\mathbf Q}_1),\ldots,\lambda_n(\hat{\mathbf Q}_1))^\top$ and $\lambda(\hat{\mathbf Q}_2)=(\lambda_1(\hat{\mathbf Q}_2),\ldots,\lambda_n(\hat{\mathbf Q}_2))^\top$ with $\lambda_1(\hat{\mathbf Q}_1)\geq\lambda_2(\hat{\mathbf Q}_1)\geq\ldots\geq\lambda_n(\hat{\mathbf Q}_1)$  and $\lambda_1(\hat{\mathbf Q}_2)\geq\lambda_2(\hat{\mathbf Q}_2)\geq\ldots\geq\lambda_n(\hat{\mathbf Q}_2)$ being the eigenvalues of $\hat{\mathbf Q}_1$ and $\hat{\mathbf Q}_2$, respectively.
\end{proposition}

\section{The Generalized Jacobi Eigenvalue Algorithm}\label{Jacobi_Method}

From \cite[Theorem 4.1]{qc11}, a dual quaternion Hermitian matrix $\hat{\mathbf{Q}}\in \hat{\mathbb{H}}^{n}$ has exactly $n$ eigenvalues and $n$ eigenvectors, in which eigenvalues are dual numbers and  eigenvectors are orthonormal vectors. In this section, in order to compute the $n$ eigenvalues and eigenvectors of $\hat{\mathbf{Q}}\in \hat{\mathbb{H}}^{n}$, we propose three Jacobi-type eigenvalue algorithms by extending the Jacobi eigenvalue algorithm to dual quaternion Hermitian matrices.

We first show that the $F$-norm of a dual quaternion matrix is invariable under unitary transformation.

\begin{theorem}\label{theorem1}
Suppose that $\hat{\mathbf Q}=(\hat{q}_{ij})\in \hat{\mathbb Q}^{n\times n}$ is a dual quaternion matrix and $\hat{\mathbf U},\hat{\mathbf V}\in \hat{\mathbb U}^{n}_2$ are unitary matrices, then
\begin{equation}
     \|  \hat{\mathbf Q}  \|_{F}^{2}= \|  \hat{\mathbf V} \hat{\mathbf Q} \hat{\mathbf U}^{\ast } \|_{F}^{2}.
\end{equation}
\end{theorem}
\begin{proof}
    From Definition \ref{def8} (iii) and (\ref{fnorm}), we have
    \begin{equation*}
        \|  \hat{\mathbf Q}  \|_{F}^{2}=\sum_{i,j=1}^{n} | \hat{q}_{ij}  |^{2}={\rm tr} (  \hat{\mathbf Q}^{\ast }  \hat{\mathbf Q}  ),
    \end{equation*}
    and
    \begin{equation*}
       \|  \hat{\mathbf V} \hat{\mathbf Q} \hat{\mathbf U}^{\ast } \|_{F}^{2}={\rm tr} ( \hat{\mathbf U} \hat{\mathbf Q}^{\ast }  \hat{\mathbf Q} \hat{\mathbf U}^{\ast } ).
    \end{equation*}
   Denote $\hat{\mathbf P} =\hat{\mathbf Q}^{\ast }\hat{\mathbf Q}$. Then there exist a unitary dual quaternion matrix $\hat{\mathbf W}$ such that $\hat{\mathbf P} =\hat{\mathbf W}\hat{\Sigma}\hat{\mathbf W}^{\ast }$, in which $\hat{\Sigma} ={\rm diag} (\hat{\lambda}_{1},\cdots ,\hat{\lambda}_{n} )$ and $\hat{\lambda}_{1},\cdots ,\hat{\lambda}_{n} $ are dual numbers.
    We rewrite $\hat{\mathbf W}=(\hat{\mathbf w}_{1},\hat{\mathbf w}_{2},\cdots,\hat{\mathbf w}_{n})$, then
    \begin{equation*}
        \begin{aligned}
            {\rm tr}( \hat{\mathbf Q}^{\ast }  \hat{\mathbf Q} )
            ={\rm tr}( \sum_{i=1}^{n}\hat{\lambda}_{i}  \hat{\mathbf w}_{i} \hat{\mathbf w}_{i}^{\ast } )
            =\sum_{i=1}^{n}\hat{\lambda}_{i} {\rm tr}(  \hat{\mathbf w}_{i} \hat{\mathbf w}_{i}^{\ast } )
            =\sum_{i=1}^{n}\hat{\lambda}_{i}.
        \end{aligned}
    \end{equation*}
    Since $\hat{\mathbf U}\hat{\mathbf W}$ is a unitary dual quaternion Hermitian matrix, then
    \begin{equation*}
        \begin{aligned}
            {\rm tr}(\hat{\mathbf U} \hat{\mathbf Q}^{\ast }  \hat{\mathbf Q} \hat{\mathbf U}^{\ast } )
            ={\rm tr}(\hat{\mathbf U} \hat{\mathbf W} \hat{\Sigma} \hat{\mathbf W}^{\ast } \hat{\mathbf U}^{\ast })
            =\sum_{i=1}^{n}\hat{\lambda}_{i}.
        \end{aligned}
    \end{equation*}
    Hence,
    \begin{equation*}
    \|  \hat{\mathbf Q}\|_{F}^{2}=\|  \hat{\mathbf V} \hat{\mathbf Q} \hat{\mathbf U}^{\ast } \|_{F}^{2}.
    \end{equation*}
\end{proof}

By Theorem \ref{theorem1}, we immediately obtain the similar result for dual quaternion Hermitian matrices.
\begin{corollary} \label{c3.2}
Suppose that $\hat{\mathbf Q}=\tilde{\mathbf Q}_{st}+\tilde{\mathbf Q}_{\mathcal I}\varepsilon\in \hat{\mathbb H}^{n}$ is a dual quaternion Hermitian matrix and $\hat{\mathbf U}\in \hat{\mathbb U}^{n}_2$ is a unitary matrices, then
\begin{equation}
    \|  \hat{\mathbf Q}\|_{F}^{2}= \|  \hat{\mathbf U} \hat{\mathbf Q} \hat{\mathbf U}^{\ast }\|_{F}^{2}.
\end{equation}
If $\tilde{\mathbf V}$ is a unitary quaternion matrix, then
\begin{equation}
    \|  \tilde{\mathbf Q}_{st} \|_{F}^{2}= \| ( \tilde{\mathbf V} \hat{\mathbf Q} \tilde{\mathbf V}^{\ast })_{st} \|_{F}^{2}\quad \text{and} \quad
 \|  \tilde{\mathbf Q}_{\mathcal I} \|_{F}^{2}= \| ( \tilde{\mathbf V} \hat{\mathbf Q} \tilde{\mathbf V}^{\ast })_{\mathcal I} \|_{F}^{2}.
\end{equation}
\end{corollary}

\subsection{Jacobi method for dual quaternion Hermitian matrices}
The Givens matrix plays an important role in Jacobi method. For the purpose of extending such a transformation to dual quaternion matrices, we intend to find the unitary dual quaternion matrix with size $2\times 2$ to digonalise a dual quaternion Hermitian matrix $\hat{\mathbf{Q}}\in \hat{\mathbb{H}}^{2}$. By simple computations, we can obtain the following two propositions which provide formulas to compute such unitary dual quaternion matrices.

\begin{proposition}\label{lemma1}
Let
$\tilde{\mathbf{Q}}=
\begin{bmatrix}
a & \tilde{c}\\
\tilde{c}^{\ast } & b
\end{bmatrix}\in \mathbb{H}^{2}$
be a quaternion Hermitian matrix. Assume that $\tilde{c}\neq 0$, define
\begin{equation}\label{tu}
    \tilde{\mathbf{U}}
    =\begin{bmatrix}
    \frac{-\tilde{c}}{\left (\left ( a-\lambda _{1}\right )^{2}+\tilde{c}^{\ast }\tilde{c} \right )^{\frac{1}{2}}} & \frac{-\tilde{c}}{\left (\left ( a-\lambda _{2}\right )^{2}+\tilde{c}^{\ast }\tilde{c} \right )^{\frac{1}{2}}}\\
    \frac{a-\lambda _{1}}{\left (\left ( a-\lambda _{1}\right )^{2}+\tilde{c}^{\ast }\tilde{c} \right )^{\frac{1}{2}}} & \frac{a-\lambda _{2}}{\left (\left ( a-\lambda _{2}\right )^{2}+\tilde{c}^{\ast }\tilde{c}\right )^{\frac{1}{2}}}
    \end{bmatrix},
\end{equation}
where $\lambda _{1},\lambda _{2}$ are the solutions of the equation $(a-x)(b-x)=\tilde{c}^{\ast }\tilde{c}$. Then, $\tilde{\mathbf{U}}$ is a unitary quaternion matrix and
$$\tilde{\mathbf{U}}^{\ast } \tilde{\mathbf{Q}} \tilde{\mathbf{U}}={\rm diag}( \lambda _{1},\lambda _{2}),$$ which is a diagonal
matrix and $\lambda _{1} \neq\lambda _{2}$.
\end{proposition}

\begin{proposition}\label{lemma2}
Let $\hat{\mathbf Q}=\tilde{\mathbf{Q}}_{st} +\tilde{\mathbf{Q}}_{\mathcal I}\varepsilon\in\hat{ \mathbb{H}}^{2}$ be a dual quaternion Hermitian matrix. Suppose that
$\tilde{\mathbf{U}}$ is a unitary quaternion matrix defined as (\ref{tu}) such that $\tilde{\mathbf{U}}^{\ast } \tilde{\mathbf{Q}}_{st} \tilde{\mathbf{U}}={\rm diag}\left ( \lambda _{1},\lambda _{2}\right )$. Denote
$\tilde{\mathbf{U}}^{\ast } \tilde{\mathbf{Q}}_{\mathcal I} \tilde{\mathbf{U}}=
\begin{bmatrix}
x & \hat{z}\\
\hat{z}^{\ast } & y
\end{bmatrix}$ and define
\begin{equation}\label{tv}
    \hat{\mathbf V}=\begin{bmatrix}
    1 & \frac{\hat{z}}{\lambda _{2}-\lambda _{1}}\varepsilon \\
    \frac{\hat{z}^{\ast }}{\lambda _{1}-\lambda _{2}}\varepsilon & 1
    \end{bmatrix}.
\end{equation}
Then, we have
$$\hat{\mathbf V}^{\ast  }\tilde{\mathbf{U}}^{\ast } \hat{\mathbf Q}^{\ast  }\tilde{\mathbf{U}} \hat{\mathbf V}={\rm diag} ( \lambda _{1}+x\varepsilon ,\lambda _{2}+y\varepsilon ),
$$ which is a diagonal dual number matrix and $\tilde{\mathbf{U}} \hat{\mathbf V}$ is a unitary dual quaternion matrix.
\end{proposition}

From the above two propositions, we define two special matrices $L(\tilde{\mathbf Q},k,l)$ associated with a quaternion Hermitian matrix $\tilde{\mathbf Q}\in \mathbb{H}^{n}$ and $J_{\hat{\mathbf Q}}(k,l)$ associated with a dual quaternion Hermitian matrix $\hat{\mathbf Q}\in \hat{\mathbb{H}}^{n}$ for any $1\leq k < l\leq n$.

\begin{definition}\label{L}
Let $\tilde{\mathbf Q}\in \mathbb{H}^{n}$ be a quaternion Hermitian matrix. For any $k,l$ with $1\leq k < l\leq n$, define $L(\tilde{\mathbf Q},k,l)=(L_{ij})\in \mathbb{U}^{n}_2$, in which $L_{ii}=1$ for all $i\ne k,l$; the submatrix formed by the intersection of its $k^{th},l^{th}$ rows and $k^{th},l^{th}$ columns is the $2\times 2$ unitary quaternion matrix as given in Proposition \ref{lemma1} which diagonalizes the submatrix formed by the intersection of $k^{th},l^{th}$ rows and $k^{th},l^{th}$ columns of $\tilde{\mathbf Q}$; and elements $L_{ij}$ in other positions are all zero.
\end{definition}

\begin{definition}\label{J}
Let $\hat{\mathbf Q}\in \hat{\mathbb{H}}^{n}$ be a dual quaternion Hermitian matrix. For any $k,l$ with $1\leq k < l\leq n$, define $J_{\hat{\mathbf Q}}(k,l)=\left ( J_{ij}\right )\in \hat{\mathbb U}^{n}_2$, where $J_{ii}=1$ for all $i\ne k,l$; the submatrix formed by the intersection of its $k^{th},l^{th}$ rows and $k^{th},l^{th}$ columns is the $2\times 2$ unitary dual quaternion matrix as given in Proposition \ref{lemma2} which can diagonalize the submatrix formed by the intersection of $k^{th},l^{th}$ rows and $k^{th},l^{th}$ columns of $\hat{\mathbf Q}$; and elements $J_{ij}$ in other positions are all zero. We say that  $J_{\hat{\mathbf Q}}(k,l)$ is a {\bf Givens matrix} for $\hat{\mathbf Q}$ with positions $k$ and $l$.
\end{definition}

With the unitary transformation given by the Givens matrix
$J_{\hat{\mathbf Q}}(k,l)$, we can {\it eliminate} the elements of $\hat{\mathbf Q}$ in positions $(k,l)$ and $(l,k)$. That is to say, we denote $J^*_{\hat{\mathbf Q}}(k,l)\hat{\mathbf Q}J_{\hat{\mathbf Q}}(k,l)=(\hat{p}_{ij})$, then we have $\hat{p}_{kl}=\hat{p}_{lk}=\hat{0}$, where $\hat{0}=\tilde{0}+\tilde{0}\varepsilon$ is the zero element of $\hat{\mathbb Q}$. In order to describe the efficiency of such eliminations,
we define two functions of $\hat{\mathbf Q}=(\hat{q}_{ij})\in \hat{\mathbb{Q}}^{n\times n}$ and $\tilde{\mathbf P}=(\tilde{p}_{ij})\in \mathbb{Q}^{n\times n}$ by
\begin{equation}
    N(\hat{ \mathbf Q}):=\| \hat{\mathbf Q} \|_{F}^{2}-\sum_{i=1}^{n}| \hat{q}_{ii} |^{2},
\quad   N(\tilde{ \mathbf P}):= \|  \tilde{\mathbf P} \|_{F}^{2}-\sum_{i=1}^{n} | \tilde{p}_{ii} |^{2}.
\end{equation}
We will find that $N(\hat{ \mathbf Q})$ is decreasing via such eliminations.

\begin{theorem}\label{lemma3}
Let $\hat{\mathbf Q}=(\hat{q}_{ij})\in \hat{\mathbb{H}}^{n}$ and $J_{\hat{\mathbf Q}}(k,l)$ be a Givens matrix of $\hat{\mathbf Q}$ with positions $k$ and $l$. Then, it holds
\begin{equation}\label{z1}
    N\left (J^{\ast }_{\hat{\mathbf Q}}(k,l)\hat{\mathbf Q}J_{\hat{\mathbf Q}}(k,l) \right )=N(\hat{ \mathbf Q})-\left | \hat{q}_{kl} \right |^{2}-\left | \hat{q}_{lk} \right |^{2}.
\end{equation}
Moreover, if $|\hat{q}_{kl}|=\underset{i\ne j}{\max}|\hat{q}_{ij}|$, then
\begin{equation*}
N( \tilde{\mathbf P}_{st})=N( \tilde{\mathbf Q}_{st})- |(\tilde{\mathbf Q}_{st})_{kl}|^{2}-|(\tilde{\mathbf Q}_{st})_{lk}|^{2}
\leq (1-\frac{2}{n(n-1)} )
N( \tilde{\mathbf Q}_{st}),
\end{equation*}
where $\tilde{\mathbf P}_{st}$ be the standard part of $J^{\ast }_{\hat{\mathbf Q}}(k,l)\hat{\mathbf Q}J_{\hat{\mathbf Q}}(k,l)$.
\end{theorem}
\begin{proof}
Denote $\hat{\mathbf P}=J_{\hat{\mathbf Q}}\left ( k,l\right )^{\ast }\hat{\mathbf Q}J_{\hat{\mathbf Q}}\left ( k,l\right )=(\hat{p}_{ij})$. From Corollary \ref{c3.2} and Definition \ref{J}, we have
\begin{align*}
N(\hat{\mathbf P})&= \| \hat{\mathbf P}  \|_{F}^{2}-\sum_{i=1}^{n}\left | \hat{p}_{ii} \right |^{2}
\\&= \| \hat{\mathbf Q}  \|_{F}^{2}-\sum_{i\ne k,l}^{}\left | \hat{q}_{ii}
\right |^{2}- ( \left | \hat{p}_{kk}\right |^{2}+\left |\hat{p}_{ll}\right |^{2}  )
\\&= \| \hat{\mathbf Q}  \|_{F}^{2}-\sum_{i\ne k,l}^{}\left | \hat{q}_{ii}\right |^{2}- ( \left | \hat{q}_{kk}\right |^{2}+\left |\hat{q}_{ll}\right |^{2} +
 \left | \hat{q}_{kl}\right |^{2}+\left |\hat{q}_{lk}\right |^{2}  )
\\&= \| \hat{\mathbf Q}  \|_{F}^{2}-\sum_{i=1}^{n}\left | \hat{q}_{ii}\right |^{2}-\left | \hat{q}_{kl}\right |^{2}-\left |\hat{q}_{lk}\right |^{2}
\\&=N(\hat{\mathbf Q})-\left | \hat{q}_{kl}\right |^{2}-\left |\hat{q}_{lk}\right |^{2}.
\end{align*}
If $|\hat{q}_{kl}|=\underset{i\ne j}{\max}|\hat{q}_{ij}|$, from (\ref{z1}), it is easy to show that
\begin{equation*}
N( \tilde{\mathbf P}_{st})=N( \tilde{\mathbf Q}_{st})- |(\tilde{\mathbf Q}_{st})_{kl}|^{2}-|(\tilde{\mathbf Q}_{st})_{lk}|^{2}
\leq (1-\frac{2}{n(n-1)} )
N( \tilde{\mathbf Q}_{st}).
\end{equation*}
Thus, we complete the proof.
\end{proof}

Clearly, by Theorem \ref{lemma3}, if the standard part of $\hat{q}_{kl}$ is not equal to zero, then
\begin{equation*}
    N\left (J^{\ast }_{\hat{\mathbf Q}}(k,l)\hat{\mathbf Q}J_{\hat{\mathbf Q}}\left ( k,l\right ) \right )<
    N(\hat{\mathbf Q}), \quad N( \tilde{\mathbf P}_{st})< N(\tilde{\mathbf Q}_{st}),
\end{equation*}
which implies that $N(\hat{ \mathbf Q})$ is decreasing under the unitary transformation given by the Givens matrix $J_{\hat{\mathbf Q}}(k,l)$.

Theorem \ref{lemma3} indicates that the main idea of the generalized Jacobi eigenvalue algorithm is to eliminate two off-diagonal elements of $\hat{\mathbf Q}$ with the largest magnitude at each iteration via the Givens matrix, which can be obtained by Propositions \ref{lemma1} and \ref{lemma2}. We describe the generalized Jacobi eigenvalue algorithm for computing eigenvalues and eigenvectors of dual quaternion Hermitian matrices as follows.

\begin{algorithm}[!ht]
\footnotesize
    \caption{The Generalized Jacobi Eigenvalue Algorithm }
    \begin{algorithmic}
    \REQUIRE  $\hat{\mathbf Q}=\tilde{\mathbf Q}_{st} +\tilde{\mathbf Q}_{\mathcal I}\varepsilon\in\hat{ \mathbb{H}}^{n}$.
    \STATE Let $\hat{\mathbf Q}^{\left ( 0\right )}=\hat{\mathbf Q}$ and $\hat{\mathbf J}=\hat{\mathbf{I}}_n$. Select an accuracy parameter $\epsilon > 0$. Set $t:=0$ and $r^{(0)}:=1$.
    \WHILE{$r^{(t)}\ge \epsilon$}
    \STATE Compute $r^{(t)}=\underset{i<j}{\max}|(\tilde{\mathbf Q}^{(t)}_{st})_{ij}|$ and $(k,l)=\underset{k<l}{\arg\max}|(\tilde {\mathbf Q}^{(t)}_{st} )_{kl} |$.
    \STATE Compute the Givens matrix $J_{\hat{\mathbf Q}^{(t)}}(k,l)$ by  Propositions \ref{lemma1} and \ref{lemma2}.
    \STATE Update $\hat{\mathbf Q}^{\left ( t+1\right )}:=J^{\ast }_{\hat{\mathbf Q}^{\left ( t\right )}}(k,l)\hat{\mathbf Q}^{\left ( t\right )}J_{\hat{\mathbf Q}^{\left ( t\right )}}(k,l)$, $\hat{\mathbf J}:=\hat{\mathbf J}J_{\hat{\mathbf Q}^{(t)}}(k,l)$ and $t:=t+1$.
    \ENDWHILE
    \STATE Denote $\lambda _{i}=(\hat{\mathbf Q}^{( t )})_{ii}$ for $i=1,\cdots ,n$.
    \STATE Let $(\hat{\mathbf V})_{ij}=\frac{(\tilde{\mathbf Q}^{(t)}_{\mathcal I} )_{ji}}{ (\tilde{\mathbf Q}^{(t)}_{st})_{ii}- (\tilde{\mathbf Q}^{(t)}_{st} )_{jj}}\varepsilon$ for $j \neq i$ and $(\hat{\mathbf V})_{ii}=\hat{1}$ for $i=1,2,\cdots,n$.
    \STATE Update $\hat{\mathbf J}:=\hat{\mathbf J}\hat{\mathbf V}$.
    \ENSURE $\left \{\lambda _{i}\right \}_{i=1}^{n}$, $\hat{\mathbf J}$.
    \end{algorithmic}
    \label{Jacobi_method_algorithm1}
\end{algorithm}

The following theorem shows that Algorithm \ref{Jacobi_method_algorithm1} terminates after finite many iterations. Moreover, the standard part of the obtained dual quaternion Hermitian matrix at the finial iteration can approximate to a diagonal dual number matrix.

\begin{theorem}\label{thmst} For any given $\hat{\mathbf Q}=\tilde{\mathbf Q}_{st} +\tilde{\mathbf Q}_{\mathcal I}\varepsilon\in\hat{ \mathbb{H}}^{n}$ and accuracy parameter $\epsilon > 0$, Algorithm \ref{Jacobi_method_algorithm1} terminates in at most $T$ iterations with $r^{(T)}< \epsilon$, where $T$ is a fixed positive integer. Moreover, suppose that the eigenvalues of $\tilde{\mathbf Q}_{st}$ are $\mu_1\ge\mu_2\ge\cdots\ge\mu_n$, then
$$\max_{1\le i\le n}|\mu_i-(\tilde{\mathbf Q}^{(T)}_{st})_{ii}|<\sqrt{n(n-1)}\epsilon.
$$
\end{theorem}
\begin{proof}
Let ${\rm diag}(\tilde{\mathbf Q}^{(t)}_{st})$ be consisted of the diagonal elements $(\tilde{\mathbf Q}^{(t)}_{st})_{ii}$ of $\tilde{\mathbf Q}^{(t)}_{st}$  with  $(\tilde{\mathbf Q}^{(t)}_{st})_{11}\ge \ldots\ge (\tilde{\mathbf Q}^{(t)}_{st})_{nn}$ at the $t$-th iteration in Algorithm \ref{Jacobi_method_algorithm1}. It follows from (\ref{H-W}) that
\begin{equation}\label{z2}
\sum_{i=1}^{n} (\mu_i-(\tilde{\mathbf Q}^{(t)}_{st})_{ii})^2\le \|\tilde{\mathbf Q}^{(t)}_{st} -
\operatorname{diag}(\tilde{\mathbf Q}^{(t)}_{st})
\|^2_{F}\le n(n-1) \underset{i\ne j}{\max}
| (\tilde{\mathbf Q}^{(t)}_{st})_{ij} |^2.
 \end{equation}
By Theorem \ref{lemma3}, we have
\begin{align}\label{z3}
N(\hat{\mathbf Q}_{st}^{\left ( t+1\right )})
&=N( \tilde{\mathbf Q}^{\left ( t\right )}_{st})-
| (\tilde{\mathbf Q}^{\left ( t\right )}_{st}  )_{kl}|^{2}-|(\tilde{\mathbf Q}^{\left ( t\right )}_{st})_{lk} |^{2}
\\
&\leq  (1-\frac{2}{n\left ( n-1\right )} )
N( \tilde{\mathbf Q}^{\left ( t\right )}_{st})
\nonumber \\&\leq \Delta^{t+1}
N( \hat{\mathbf Q}_{st}^{\left ( 0\right )}), \nonumber
\end{align}
where
$$\Delta=1-\frac{2}{n\left ( n-1\right )}.$$
Taking
$$T=\left\lceil \frac{\log(\epsilon/N( \hat{\mathbf Q}_{st}^{(0)}))}{\log \Delta}\right\rceil,$$
by the definition of $r^{(t)}$ in Algorithm \ref{Jacobi_method_algorithm1}, (\ref{z2}) and (\ref{z3}) implies that $r^{(T)}<\epsilon$. That is, the algorithm terminates in at most $T$ iterations.
Moreover, it follows from (\ref{z2}) that
$$
\max_{1\le i\le n}|\mu_i-(\tilde{\mathbf Q}^{(T)}_{st})_{ii}|<\sqrt{n(n-1)}\epsilon
$$
holds.
\end{proof}


From Theorem \ref{thmst}, we make sure that Algorithm \ref{Jacobi_method_algorithm1} can provide $\epsilon$-approximation to the standard part of eigenvalue of $\hat{\mathbf Q}$ after finite may iterations. In order to prove Algorithm \ref{Jacobi_method_algorithm1} can also provide $\epsilon$-approximation to the dual part of eigenvalue of $\hat{\mathbf Q}$,  we need the following two lemmas.
\begin{lemma}\label{lemma4}
Let $\hat{\mathbf Q}=\tilde{\mathbf Q}_{st} +\tilde{\mathbf Q}_{\mathcal I}\varepsilon \in\hat{ \mathbb{H}}^{n}$ be a dual quaternion Hermitian matrix with $\tilde{\mathbf Q}_{st}=(\tilde{q}_{1,ij})$ and $\tilde{\mathbf Q}_{\mathcal I}=(\tilde{q}_{2,ij})$.
Suppose that $\tilde{q}_{1,ij}=\tilde{0}$ and $\tilde{q}_{1,ii} \neq \tilde{q}_{1,jj}$ for $i,j=1,\cdots,n$ with $i\neq j$, then $\tilde{q}_{1,ii}+\tilde{q}_{2,ii}\varepsilon$ is an eigenvalue of
$\hat{\mathbf Q}$ with its corresponding eigenvector $\hat{\mathbf v}_i$ for $i=1,\ldots,n$, where the $i$th component is $(\hat{\mathbf v}_i)_i=\hat{1}$, and the $j$th component is $(\hat{\mathbf v}_i)_j=\dfrac{\tilde{q}_{2,ji}}{\tilde{q}_{1,ii}-\tilde{q}_{1,jj}}\varepsilon$ for $j \neq i$.
\end{lemma}

Lemma \ref{lemma4} can be easily obtained by direct calculation. For any given dual quaternion Hermitian matrix $\hat{\mathbf Q}=\tilde{\mathbf Q}_{st} +\tilde{\mathbf Q}_{\mathcal I}\varepsilon$, if $\tilde{\mathbf Q}_{st}$ is a diagonal matrix with different diagonal elements, then Lemma \ref{lemma4} shows that its diagonal elements are exactly $n$ eigenvalues of
$\hat{\mathbf Q}$ and the corresponding eigenvectors are explicitly given.

\begin{lemma}\label{lemma5}
Let $\delta>0$ and $\tilde{\mathbf Q}=(\tilde{q}_{ij})\in\mathbb{H}^{n}$ be a quaternion Hermitian matrix with $\underset{i\ne j}{\max}| \tilde{q}_{ij} |\le\delta $. Suppose that $\tilde{\mathbf Q}$ has $n$ different eigenvalues $\mu_1,\mu_2,\ldots,\mu_n$ and the corresponding eigenvector $\tilde{\mathbf v}_1,\ldots,\tilde{\mathbf v}_n$ satisfying $\tilde{\mathbf v}_i^*\tilde{\mathbf v}_i=\tilde{1}$.
Define $c=\underset{i\ne j}{\min}|\mu_i-\mu_j |$ and $$\tilde{\mathbf V}=\left (\tilde{\mathbf v}_1\frac{(\tilde{\mathbf v}_1)_1^*}{\left | (\tilde{\mathbf v}_1)_1\right |},\tilde{\mathbf v}_2\frac{(\tilde{\mathbf v}_2)_2^*}{\left | (\tilde{\mathbf v}_2)_2 \right |},\cdots,\tilde{\mathbf v}_n\frac{(\tilde{\mathbf v}_n)_n^*}{\left | (\tilde{\mathbf v}_n)_n \right |} \right ) .$$
Then,
$$\tilde{\mathbf V}^*\tilde{\mathbf Q}\tilde{\mathbf V}=\operatorname{diag}(\mu_1,\mu_2,\cdots,\mu_n).$$
Moreover, if $c\ge 2\sqrt{n(n-1)}\delta$, then $$\underset{1\le i,j\le n}{\max}|(\tilde{\mathbf{I}}_n - \tilde{\mathbf V})_{ij}| \le \frac{2\delta}{c}.$$
\end{lemma}

\begin{proof}
Since $\tilde{\mathbf v}_i\frac{(\tilde{\mathbf v}_i)_i^*}{\left | (\tilde{\mathbf v}_i)_i \right |}$ is also an eigenvector of $\tilde{\mathbf Q}$ with respect to $\mu_i$ and $\|\tilde{\mathbf v}_i\frac{(\tilde{\mathbf v}_i)_i^*}{\left | (\tilde{\mathbf v}_i)_i \right |}\|=1$, then $\tilde{\mathbf V}^*\tilde{\mathbf Q}\tilde{\mathbf V}=\operatorname{diag}(\mu_1,\mu_2,\cdots,\mu_n) $.

It follows from Proposition \ref{lemma0} and $\underset{i\ne j}{\max}\left | \tilde{q}_{ij} \right |\le\delta $ that there is a permutation $\sigma$ of $\{1,\ldots,n\}$ such that
$$\sum_{i=1}^{n} (\mu_{\sigma(i)}-\tilde{q}_{ii})^2 \le n(n-1) \delta ^2.$$
 Without loss of generality, let $\sigma(i)=i$. Since $\tilde{\mathbf v}_i$ is an eigenvector of $\tilde{\mathbf Q}$ with respect to eigenvalue $\mu_i$, then
 $$(\tilde{\mathbf Q}-\mu_i\tilde{\mathbf{I}}_n)\tilde{\mathbf v}_i=\tilde{0}.$$
For $1\le k\le n$, we have
$$\sum_{j\ne k}^{} \tilde{q}_{kj}(\tilde{\mathbf v}_i)_j+
(\tilde{q}_{kk}-\mu_i)(\tilde{\mathbf v}_i)_k=\tilde{0}.$$
This yields
 $$\left | (\tilde{q}_{kk}-\mu_i)(\tilde{\mathbf v}_i)_k \right | = | \sum_{j\ne k}^{} \tilde{q}_{kj}(\tilde{\mathbf v}_i)_j  |
\le \delta \left \| \tilde{\mathbf v}_i \right \|= \delta.$$
Suppose that $c=t\sqrt{n(n-1)}\delta$, then we have for $k\ne i$, $$\delta \ge \left | (\tilde{q}_{kk}-\mu_i)
(\tilde{\mathbf v}_i)_k \right | \ge(c-\sqrt{n(n-1)}\delta)\left | (\tilde{\mathbf v}_i)_k \right |
,$$ i.e.,$$\left | (\tilde{\mathbf v}_i)_k \right | \le \frac{\delta}{(c-\sqrt{n(n-1)
}\delta)}=\frac{t\delta}{(t-1)c},$$
which implies that when $t\ge 2$, for $k\ne i$,
$$|(\tilde{\mathbf v}_i)_k \frac{(\tilde{\mathbf v}_i)_i^*}{\left | (\tilde{\mathbf v}_i)_i \right |}|=\left | (\tilde{\mathbf v}_i)_k \right | \le \frac{t\delta}{(t-1)c}\le \frac{2\delta}{c}.$$
Then,
\begin{align*}
1- (\tilde{\mathbf v}_i)_i\frac{(\tilde{\mathbf v}_i)_i^*}{\left | (\tilde{\mathbf v}_i)_i \right |} &=1-(1-\sum_{j\ne i}^{}\left | (\tilde{\mathbf v}_i)_j \right | ^2)^\frac{1}{2}
\\& \le 1-(1-\frac{(n-1)\delta ^2}{(c-\sqrt{n(n-1)}\delta)^2} )^\frac{1}{2}
\\& =1-(1-\frac{1}{n(t-1)^2} )^\frac{1}{2}.
\end{align*}

When $t\ge 2$, it holds
\begin{align*}
&1-(1-\frac{1}{n(t-1)^2} )^\frac{1}{2}\le \frac{2\delta}{c}\\ \Leftrightarrow \quad &
1-(1-\frac{1}{n(t-1)^2} )^\frac{1}{2}\le \frac{2}{\sqrt{n(n-1)}t}
\\ \Leftrightarrow \quad &\frac{1}{n(t-1)^2}+\frac{4}{n(n-1)t^2}\le \frac{4}{\sqrt{n(n-1)}t}
\\ \Leftrightarrow \quad &\frac{n-1}{(t+\frac{1}{t}-2)}+\frac{4}{t}\le 4\sqrt{n(n-1)},
\end{align*}
and
$$\frac{n-1}{(t+\frac{1}{t}-2)}+\frac{4}{t}\le \frac{n-1}{(2+\frac{1}{2}-2)}+\frac{4}{2}\le 2n\le 4\sqrt{n(n-1)},$$
hence, we have
$$1- (\tilde{\mathbf v}_i)_i\frac{(\tilde{\mathbf v}_i)_i^*}{\left | (\tilde{\mathbf v}_i)_i \right |}  \le \frac{2\delta}{c}.$$
Thus,
$$\underset{1\le i,j\le n}{\max}|(\tilde{\mathbf{I}}_n - \tilde{\mathbf V})_{ij}| \le \frac{2\delta}{c}.$$
This completes the proof.
\end{proof}

Given a dual quaternion Hermitian matrix $\hat{\mathbf Q}$ with eigenvalues $\hat{\lambda}_i=\mu_i+\eta_i\varepsilon$, $i=1,\ldots,n$. Theorem \ref{thmst} shows that Algorithm \ref{Jacobi_method_algorithm1} can provide $\epsilon$-approximation of $\mu_i$ in at most $T$ iterations. Define $$c=\underset{i\ne j}{\min}|\mu_i-\mu_j|.$$ Then, the following theorem shows that it also provides $\epsilon$-approximation of $\eta_i$.

\begin{theorem}\label{thmd}
Let $T$ be given in Theorem \ref{thmst} and $\epsilon>0$ be given accuracy parameter. Suppose that $\hat{\mathbf Q}$ has $n$ simple eigenvalues, $c\ge 8(n+1)\epsilon$, and the non-diagonal elements of the dual part of $\tilde{\mathbf Q}^{(t)}_{\mathcal I}$ are bounded for $1\le t\le T$, then there exists a constant $C>0$ such that $$\max_{1\le i\le n}|  \eta_i- (\tilde{\mathbf Q}^{(T)}_{\mathcal I})_{ii} |
\le C\epsilon.$$
That is, Algorithm \ref{Jacobi_method_algorithm1} can provide $\epsilon$-approximation of $\eta_i$ in at most $T$ iterations.
\end{theorem}
\begin{proof}
According to Algorithm \ref{Jacobi_method_algorithm1},
$\underset{i\ne j}{\max} | (\tilde{\mathbf Q}^{(T)}_{st})_{ij}|\le \epsilon$. By Lemma \ref{lemma5},
there exist a unitary quaternion matrix $\tilde{\mathbf V}$, which diagonalizes the $\tilde{\mathbf Q}_{st}^{\left ( T\right )}$, and since $c=8(n+1)\epsilon\ge 2\sqrt{n(n-1)}\epsilon$, it holds $\underset{ij}{\max}|(\tilde{\mathbf{I}}_n- \tilde{\mathbf V})_{ij}| \le \frac{2\epsilon}{c}$.

Since $\tilde{\mathbf V}^*\tilde{\mathbf Q}^{(T)}_{st}\tilde{\mathbf V}$ is a diagonal matrix and
$$\tilde{\mathbf V}^*\hat{\mathbf Q}^{(T)}\tilde{\mathbf V}=\tilde{\mathbf V}^*\tilde{\mathbf Q}^{(T)}_{st}\tilde{\mathbf V}+\tilde{\mathbf V}^*\tilde{\mathbf Q}^{(T)}_{\mathcal I}\tilde{\mathbf V},$$ by Lemma \ref{lemma4}, the diagonal elements of
$\tilde{\mathbf V}^*\hat{\mathbf Q}^{(T)}\tilde{\mathbf V}$ are the eigenvalues of $\hat{\mathbf Q}^{(T)}$, then, $\eta_i=\tilde{\mathbf v}_i^*\tilde{\mathbf Q}^{(T)}_{\mathcal I}\tilde{\mathbf v}_i$, where $\tilde{\mathbf v}_i$ is the $i^{th}$ column of $\tilde{\mathbf V}$. Hence,
\begin{align*}
|\eta _i- (\tilde{\mathbf Q}^{(T)}_{\mathcal I})_{ii} |
&= |\tilde{\mathbf v}_i^*\tilde{\mathbf Q}^{(T)}_{\mathcal I}\tilde{\mathbf v}_i
-(\tilde{\mathbf Q}^{(T)}_{\mathcal I})_{ii} |
\\&\le |(\tilde{\mathbf v}_i^*-e_i^T)\tilde{\mathbf Q}^{(T)}
_{\mathcal I}(\tilde{\mathbf v}_i-e_i) |+2 |e_i^T\tilde{\mathbf Q}^{(T)}
_{\mathcal I}(\tilde{\mathbf v}_i-e_i) |
\\&\le \frac{4\epsilon^2}{c^2} \sum_{1\le j,k\le n}  |(\tilde{\mathbf Q}^{(T)}_{\mathcal I}) _{jk}
  | +\frac{4\epsilon}{c}\underset{j}{\max}\sum_{k} |(\tilde{\mathbf Q}^{(T)}_{\mathcal I}) _{jk}
 |
\\&\le \frac{4\epsilon}{c} \left ( n^2\underset{j\ne k}{\max}
 |(\tilde{\mathbf Q}^{(T)}_{\mathcal I}) _{jk}  |
+(n+1)\underset{j}{\max} |(\tilde{\mathbf Q}^{(T)}_{\mathcal I}) _{jj}  | \right ),
\end{align*}
which implies,
 $$| (\tilde{\mathbf Q}^{(T)}_{\mathcal I})_{ii}  | \le  | \eta _i  |+\frac{4\epsilon}{c}\left ( n^2\underset{j\ne k}{\max}
 |(\tilde{\mathbf Q}^{(T)}_{\mathcal I}) _{jk}  |
+(n+1)\underset{j}{\max} |(\tilde{\mathbf Q}^{(T)}_{\mathcal I}) _{jj}  | \right ).$$
Thus, $$(1-\frac{4(n+1)\epsilon}{c})\underset{i}{\max} | (\tilde{\mathbf Q}^{(T)}_{\mathcal I})_{ii} |\le \underset{i}{\max} | \eta _i  |+\frac{4n^2\epsilon}{c}\underset{j\ne k}{\max} |(\tilde{\mathbf Q}^{(T)}_{\mathcal I}) _{jk}  |.$$
It follows from $c\ge 8(n+1)\epsilon$ that
$$\underset{i}{\max} | (\tilde{\mathbf Q}^{(T)}_{\mathcal I})_{ii} |
\le 2\underset{i}{\max} | \eta _i  |+n\underset{j\ne k}{\max} |(\tilde{\mathbf Q}^{(T)}_{\mathcal I}) _{jk} |.$$
Since $\underset{j\ne k}{\max} |(\tilde{\mathbf Q}^{(T)}_{\mathcal I}) _{jk} |$ is bounded, then there exists a constant $R>0$ such that
$$\underset{j\ne k}{\max}|(\tilde{\mathbf Q}^{(T)}_{\mathcal I}) _{jk}  |\le R.$$ Hence, we have for $i=1,\ldots, n$,
$$|  \eta _i- (\tilde{\mathbf Q}^{(T)}_{\mathcal I})_{ii}|
\le  \frac{4}{c}( (2n^2+n)R +2(n+1)\underset{i}{\max} | \eta _i  | ) \epsilon.$$
Taking
$$C=\frac{4}{c}( (2n^2+n)R +2(n+1)\underset{i}{\max} | \eta _i  | ),$$ it easily holds
$$\max_{1\le i\le n}|  \eta_i- (\tilde{\mathbf Q}^{(T)}_{\mathcal I})_{ii} |
\le C\epsilon.$$ Hence,  Algorithm \ref{Jacobi_method_algorithm1} can provide $\epsilon$-approximation for $\eta_i$ in at most $T$ iterations.
\end{proof}

To accelerate the efficiency of Algorithm \ref{Jacobi_method_algorithm1}, we adopt a new elimination strategy which eliminates the non-diagonal elements with exceeding a specified threshold and gradually reduces such a threshold. This approach can ensure the magnitudes of all non-diagonal elements become smaller than the threshold. This strategy can address the time-consuming process of identifying the nondiagonal element with the maximum magnitude in Algorithm \ref{Jacobi_method_algorithm1}. We describe the strategy in Algorithm \ref{Jacobi_method_algorithm2}. Its performance will be shown in Section \ref{Experiments}.

\begin{algorithm}
\footnotesize
    \caption{The Accelerated Jacobi Eigenvalue Method}
    \begin{algorithmic}
    \REQUIRE $\hat{\mathbf Q}=\tilde{\mathbf Q}_{st} +\tilde{\mathbf Q}_{\mathcal I}\varepsilon\in\hat{ \mathbb{H}}^{n}$ and parameter $\delta>0$, $\rho\in(0,1)$, $\eta>0$.
    \STATE Set $t:=1$, $\hat{\mathbf J}:=\hat{\mathbf I}_n$, and $\delta^{(1)}:=\delta$.
    \WHILE{$\delta^{(t)}\ge \eta$}
    \WHILE{there exists $k<l$ satisfying $ | (\tilde {\mathbf Q}_{st}  )_{kl} |\ge \delta^{(t)}$ }
    \STATE Compute $J_{\hat{\mathbf Q}}( k,l)$ by Propositions \ref{lemma1} and \ref{lemma2}.
    \STATE Update $\hat{\mathbf Q}=J_{\hat{\mathbf Q}}^{\ast }( k,l)\hat{\mathbf Q}J_{\hat{\mathbf Q}} ( k,l )$ and $\hat{\mathbf J}=\hat{\mathbf J}J_{\hat{\mathbf Q}}$
    \ENDWHILE
    \STATE Update $\delta^{(t+1)}:=\rho\delta^{(t)}$ and $t:=t+1$.
    \ENDWHILE
    \STATE Denote $\lambda _{i}=(\hat{\mathbf Q})_{ii}$, $i=1,\ldots ,n$.
    \STATE Let $(\hat{\mathbf V})_{ij}=\frac{(\tilde{\mathbf Q}_{\mathcal I}  )_{ji}}{ (\tilde{\mathbf Q}_{st} )_{ii}- (\tilde{\mathbf Q}_{st} )_{jj}}\varepsilon$ for $j \neq i$ and $(\hat{\mathbf V})_{ii}=\hat{1}$ for $i=1,\ldots,n$.
    \STATE Update $\hat{\mathbf J}=\hat{\mathbf J}\hat{\mathbf V}$.
    \ENSURE $\{\lambda _{i}\}_{i=1}^{n}$ and $\hat{\mathbf J}$.
    \end{algorithmic}
    \label{Jacobi_method_algorithm2}
\end{algorithm}

\subsection{Improved Jacobi method}

Given a dual quaternion Hermitian matrix $\hat{\mathbf Q}$, if $\hat{\mathbf Q}$ has an eigenvalue with multiplicity greater than $1$, then Algorithm \ref{Jacobi_method_algorithm1} may not converge. To address this question, we propose an improved Jacobi method which repeats the elimination process given in Propositions  \ref{lemma1} and \ref{lemma2} and relies on the fact that after finite many iterations Algorithm \ref{Jacobi_method_algorithm1} can output a solution with $\epsilon$-accuracy. The improved Jacobi method includes three steps, the first step is to diagonalize the standard part of the dual quaternion Hermitian matrix, then divide the dual part of the matrix into blocks based on whether the eigenvalues of the standard part of the matrix are the same, the remaining two steps process non-diagonal blocks and then diagonal blocks, separately.
We describe the details in  Algorithm \ref{Jacobi_method_algorithm_3}.

\vspace{0.1in}
\begin{breakablealgorithm}
\footnotesize
    \label{Jacobi_method_algorithm_3}
    \caption{Three-step Jacobi Eigenvalue Algorithm (3SJacobi)}
    \begin{algorithmic}
    \REQUIRE $\hat{\mathbf Q}=\tilde{\mathbf Q}_{st} +\tilde{\mathbf Q}_{\mathcal I}\varepsilon\in\hat{ \mathbb{H}}^{n}$.
     Parameters $\delta,\delta_1>0$, $\rho\in(0,1)$, $\eta>0$, and $\gamma =\sqrt{2n(n-1)}\eta$.
    \STATE {\bf STEP 1:}
    \STATE Set $t:=1, m:=0$, $\hat{\mathbf J}_1^{(0)}:=\hat{\mathbf I}_n$, $\delta^{(1)}:=\delta$, $\hat{\mathbf Q}^{(0)}:=\hat{\mathbf Q}$.
    \WHILE{$\delta^{(t)}\ge\eta$}
    \WHILE{there exists $i_m<j_m$ satisfying $|(\tilde {\mathbf Q}^{(m)}_{st}  )_{i_mj_m} |\ge \delta^{(t)}$ }
    \STATE Compute  $L^{(m)}=L(\tilde{\mathbf Q}^{(m)}_{st},i_m,j_m)$ by  Proposition \ref{lemma1}.
    \STATE Update $\hat{\mathbf Q}^{(m+1)}:=(L^{(m)})^{\ast}\hat{\mathbf Q}^{(m)}L^{(m)}$, $\hat{\mathbf J}_1^{(m+1)}:=\hat{\mathbf J}_1^{(m)}L^{(m)}$, and $m:=m+1$.
    \ENDWHILE
    \STATE Update $\delta^{(t+1)}:=\rho\delta^{(t)}$ and $t:=t+1$.
    \ENDWHILE
    \STATE Output: $\left \{ \hat{\mathbf Q}^{(m)}=\tilde{\mathbf Q}_{st}^{(m)} +\tilde{\mathbf Q}^{(m)}_{\mathcal I}\varepsilon \right \} ^{M_1}_{m=1}$, $\hat{\mathbf J}_1^{(M_1)}$.
    \STATE {\bf STEP 2:}
    \STATE Set $m:=M_1$, $\hat{\mathbf J}_2^{(0)}:=\hat{\mathbf I}_n$.
    \FOR{$i=1:n-1$}
    \FOR{$j=i+1:n$}
    \IF{$|( \tilde {\mathbf Q}^{(M_1)}_{st} )_{ii}- ( \tilde {\mathbf Q}^{(M_1)}_{st}  )_{jj}|> \gamma $}
    \STATE Let $\hat{\mathbf P}^{(m)}=\operatorname{diag} (\tilde{\mathbf Q}^{(m)}_{st}) +\tilde{\mathbf Q}^{(m)}_{\mathcal I}\varepsilon$ and $(i_m,j_m)=(i,j)$.
    \STATE Compute the Givens matrix $J^{(m)}=J_{\hat{\mathbf P}^{(m)}}\left (i,j\right )$ by Proposition \ref{lemma2}.
    \STATE Update $\hat{\mathbf Q}^{(m+1)}:=(J^{(m)})^{\ast }\hat{\mathbf Q}J^{(m)}$, $\hat{\mathbf J}_2^{(m+1-M_1)}:=\hat{\mathbf J}_2^{(m-M_1)}J^{(m)}$, and $m:=m+1$.
    \ENDIF
    \ENDFOR
    \ENDFOR
    \STATE Output: $\left \{ \hat{\mathbf Q}^{(m)}=\tilde{\mathbf Q}_{st}^{(m)} +\tilde{\mathbf Q}^{(m)}_{\mathcal I}\varepsilon \right \} ^{M_1+M_2}_{m=M_1+1}$, $\hat{\mathbf J}_2^{(M_2)}$.
    \STATE {\bf STEP 3:}
    \STATE Set $t:=1$, $m:=M_1+M_2$, $\hat{\mathbf J}_3^{(0)}:=\hat{\mathbf I}_n$, and $\delta^{(1)}_1:=\delta_1$.
    \WHILE{$\delta^{(t)}_1\ge\eta$}
    \WHILE{there exists $i_m<j_m$ satisfying
    $$ |(\tilde {\mathbf Q}^{(m)}_{\mathcal I} )_{i_mj_m} |\ge \delta^{(t)}_1,\quad | ( \tilde {\mathbf Q}^{(M_1)}_{st} )_{ii}-( \tilde {\mathbf Q}^{(M_1)}_{st} )_{jj} |\le \gamma $$ }
    \STATE Compute  $L^{(m)}=L(\tilde{\mathbf Q}^{(m)}_{\mathcal I},i_m,j_m)$ by Proposition \ref{lemma1}.
    \STATE Update $\hat{\mathbf Q}^{(m+1)}:=(L^{(m)})^{\ast}\hat{\mathbf Q}^{(m)}L^{(m)}$, $\hat{\mathbf J}_3^{(m+1-M_1-M_2)}:=\hat{\mathbf J}_3^{(m-M_1-M_2)}L^{(m)}$, and $m:=m+1$.
    \ENDWHILE
    \STATE Update $\delta^{(t+1)}_1:=\rho\delta^{(t)}_1$ and $t:=t+1$.
    \ENDWHILE
    \STATE Output: $\left \{ \hat{\mathbf Q}^{(m)}=\tilde{\mathbf Q}_{st}^{(m)} +\tilde{\mathbf Q}^{(m)}_{\mathcal I}\varepsilon \right \} ^{M_1+M_2+M_3}_{m=M_1+M_2+1}$, $\hat{\mathbf J}_3^{(M_3)}$.
    \STATE Let $M=M_1+M_2+M_3$, $\lambda _{i}=(\hat{\mathbf Q}^{(M)})_{ii}$ for $i=1,\ldots ,n$.
    \STATE Compute $\hat{\mathbf J}=\hat{\mathbf J}_1^{(M_1)}\hat{\mathbf J}_2^{(M_2)}\hat{\mathbf J}_3^{(M_3)}$.
    \ENSURE $\left \{\lambda _{i}\right \}_{i=1}^{n}$, $\hat{\mathbf J}$.
    \end{algorithmic}
\end{breakablealgorithm}\vspace{0.1in}

For a given dual quaternion Hermitian matrix $\hat{\mathbf Q}=\tilde{\mathbf Q}_{st}+\tilde{\mathbf Q}_{\mathcal I}\in\hat{ \mathbb{H}}^{n}$, we use  Algorithm \ref{Jacobi_method_algorithm_3} to compute its eigenvalue. In order to analyze the convergence of Algorithm \ref{Jacobi_method_algorithm_3}, we denote
$$
r_m=\frac{2}{ | (\tilde{\mathbf Q}^{(M_1)}_{st})_{i_mi_m}-(\tilde{\mathbf Q}^{(M_1)}_{st})_{j_mj_m} |},
$$
and
\begin{equation}\label{albe}
\alpha=\sum_{k=M_1}^{M_1+M_2-1} r_k,\quad \beta=\prod_{k=M_1}^{M_1+M_2-1} (1+r_k\eta),\quad \kappa=1+\alpha\beta\eta.
\end{equation}
It holds the following theorem.
\begin{theorem}\label{theoreM_2}
Suppose that the eigenvalues of $\tilde{\mathbf Q}_{st}$ are
$\eta_1>\eta_2>\cdots >\eta_p$ with  multiplicities $t_i\ (i=1,\ldots,p)$, respectively.
Then Algorithm \ref{Jacobi_method_algorithm_3} terminates after at most $T$ iterations with
\begin{equation}\label{wucha1}
\max_{i\ne j}|(\tilde{\mathbf Q}_{st}^{(M)})_{ij}| \le 3n\eta,\quad
\max_{i\ne j}|(\tilde{\mathbf Q}_{\mathcal I}^{(M)})_{ij}| \le \max\{1, h_1\alpha \beta\| \tilde{\mathbf Q}_{\mathcal I}\|_F \} \eta,
\end{equation}
and
\begin{equation}\label{wucha2}
\max_{i} |(\tilde{\mathbf Q}_{\mathcal I}^{(M)})_{ii}|\le h_1\kappa\| \tilde{\mathbf Q}_{\mathcal I}\|_F,
\end{equation}
where $h_1=\max\{t_i:\ i=1,\ldots,p\}$, $h_2=\sum_{i=1}^{p}t_i^2$, and
$$T=\frac{\| \tilde{\mathbf Q}_{st}\|_F^2}{2\delta^2} +
\frac{n^2}{2}\left ( 1+\frac{\left \lceil  \log_{\rho}{\frac{\eta}{\delta}}\right\rceil}{\rho^2} \right ) +
\frac{h_2}{2}\left ( \frac{\kappa^2
\| \tilde{\mathbf Q}_{\mathcal I} \|_F^2}{\delta_1^2}+\frac{\left \lceil \log_{\rho}{\frac{\eta}{\delta_1}}\right \rceil}{\rho^2}
 \right ).$$
\end{theorem}

\begin{proof}

In the process of STEP 1, by Theorem \ref{lemma3}, $N(\tilde{\mathbf Q}_{st}^{(m)})$ is strictly decreasing and $|(\tilde{\mathbf Q}_{st}^{(M_1)})_{ij}|\le \eta$ holds for all $i\ne j$. Suppose that  there are $m_i$ iterations during the loop ``$t=i$" in STEP 1. In the loop ``$t=i$", $N(\tilde{\mathbf Q}_{st}^{(m)})$ will decrease at least $2(\delta^{(i)})^2$ for every iteration. Thus, $m_i\le \frac{N(\tilde{\mathbf Q}_{st}^{(m)})}{2(\delta^{(i)})^2} $. When $i>1$, at the end of the loop ``$t=i-1$", we have $N(\tilde{\mathbf Q}_{st}^{(m)})<n(n-1)(\delta^{(i-1)})^2$, then $$m_i\le \frac{n(n-1)}{2} (\frac{\delta^{(i-1)}}{\delta^{(i)}} )^2=\frac{n(n-1)}{2\rho^2}\le \frac{n^2}{2\rho^2}.$$ This suggests that the number of iterations in STEP 1 is at most
\begin{equation}\label{eqt1}
T_1\doteq\frac{\|\tilde{\mathbf Q}_{st}\|_F^2}{2\delta^2} +\frac{n^2}{2\rho^2}\left \lceil \log_{\rho}{\frac{\eta}{\delta}}  \right \rceil.
\end{equation}

In the process of STEP 2, since $\hat{\mathbf P}^{(m)}_{st}=\operatorname{diag} (\tilde{\mathbf Q}^{(m)}_{st})$ is a diagonal matrix,
we have $J^{(m)}_{st}=\tilde{\mathbf I}$ and
$$(J^{(m)}_{\mathcal I})_{ij}=\begin{cases}
\frac{(\tilde{\mathbf Q}^{(m)}_{\mathcal I})_{ij}}
{(\tilde{\mathbf Q}^{(m)}_{st})_{jj}
 -(\tilde{\mathbf Q}^{(m)}_{st})_{ii}} ,& \text{if} ~~ \left \{ i,j \right \} = \left \{ i_m,j_m \right \},\\
 \tilde{0}, &\text{if} ~~\left \{ i,j \right \} \ne \left \{ i_m,j_m \right \}.
\end{cases}.$$
Then
\begin{align*}
\hat{\mathbf Q}^{(m+1)}&=(J^{(m)})^{\ast }\hat{\mathbf Q}^{(m)}(J^{(m)})=(\tilde{\mathbf I}+(J^{(m)}_{\mathcal I})^{\ast }\varepsilon )\hat{\mathbf Q}^{(m)}(\tilde{\mathbf I}+J^{(m)}_{\mathcal I}\varepsilon )
\\&=\tilde{\mathbf Q}^{(m)}_{st}+(\tilde{\mathbf Q}^{(m)}_{st}J^{(m)}_{\mathcal I}+(J^{(m)}_{\mathcal I})^{\ast }
\tilde{\mathbf Q}^{(m)}_{st}+\tilde{\mathbf Q}^{(m)}_{\mathcal I})\varepsilon .
\end{align*}
Denote $S_m=\operatorname{diag} (\tilde{\mathbf Q}_{st}^{(m)}) $ , $R_m=\tilde{\mathbf Q}_{st}^{(m)}-\operatorname{diag} (\tilde{\mathbf Q}_{st}^{(m)})$ and
$D_m=J^{(m)}_{\mathcal I}$. Then
$$\tilde{\mathbf Q}^{(m+1)}_{\mathcal I}=(S_m+R_m)D_m+D_m^{\ast }(S_m+R_m)+\tilde{\mathbf Q}^{(m)}_{\mathcal I},$$
and $|(R_m)_{ij}|\le \eta$.
Based on the form of $J^{(m)}_{\mathcal I}$, it holds
$$(S_mD_m+D_m^{\ast }S_m)_{ij}=\begin{cases}
-(\tilde{\mathbf Q}^{(m)}_{\mathcal I})_{ij}, &\text{if} ~~\left \{ i,j \right \} = \left \{ i_m,j_m \right \},\\
\tilde{0},&\text{if} ~~\left \{ i,j \right \} \ne \left \{ i_m,j_m \right \}.
\end{cases} $$
Hence, if $\left \{ i,j \right \} \ne \left \{ i_m,j_m \right \}$, we have
\begin{align*}
|(\tilde{\mathbf Q}^{(m+1)}_{\mathcal I})_{ij}|
&=|((S_m+R_m)D_m+D_m^{\ast }(S_m+R_m)+\tilde{\mathbf Q}^{(m)}_{\mathcal I})_{ij}|
\\&=|(R_mD_m+D_m^{\ast }R_m+\tilde{\mathbf Q}^{(m)}_{\mathcal I})_{ij}|
\\&\le \frac{2|(\tilde{\mathbf Q}^{(m)}_{\mathcal I})_{i_mj_m}|\eta}
{| (\tilde{\mathbf Q}^{(m)}_{st})_{i_mi_m}
 -(\tilde{\mathbf Q}^{(m)}_{st})_{j_mj_m}  |} +|(\tilde{\mathbf Q}^{(m)}_{\mathcal I})_{ij}|
 \\& =r_m\eta|(\tilde{\mathbf Q}^{(m)}_{\mathcal I})_{i_mj_m}|+|(\tilde{\mathbf Q}^{(m)}_{\mathcal I})_{ij}|.
\end{align*}
If $\left \{ i,j \right \} = \left \{ i_m,j_m \right \}$, we have
\begin{align*}
|(\tilde{\mathbf Q}^{(m+1)}_{\mathcal I})_{i_mj_m}|
&=|((S_m+R_m)D_m+D_m^{\ast }(S_m+R_m)+\tilde{\mathbf Q}^{(m)}_{\mathcal I})_{i_mj_m}|
\\&=|(R_mD_m+D_m^{\ast }R_m)_{i_mj_m}|
\le \frac{2|(\tilde{\mathbf Q}^{(m)}_{\mathcal I})_{i_mj_m}|\eta}
{| (\tilde{\mathbf Q}^{(m)}_{st})_{i_mi_m}
 -(\tilde{\mathbf Q}^{(m)}_{st})_{j_mj_m} | }
\\&=r_m\eta|(\tilde{\mathbf Q}^{(m)}_{\mathcal I})_{i_mj_m}|.
\end{align*}
Denote
$$K_1=\left \{\{ i_m,j_m \}  \right \}_{m=M_1}^{M_1+M_2-1},\quad K_2=\left \{\{ i,j \} |1\le i < j \le n \right \}\backslash K_1.$$
Then, we have
$$\underset{\left \{ i,j \right \} \in K_1}{\max} |(\tilde{\mathbf Q}^{(m+1)}_{\mathcal I})_{ij}|
\le (1+r_m\eta)\underset{\left \{ i,j \right \} \in K_1}{\max} |(\tilde{\mathbf Q}^{(m)}_{\mathcal I})_{ij}|.$$
Denote $K_3=\{ M_1, M_1+1, \ldots, M_1+M_2-1\}$, it holds
$$\max_{\substack{\left \{ i,j  \right \} \in K_1 \\ m\in K_3}} |(\tilde{\mathbf Q}^{(m)}_{\mathcal I})_{ij}|
\le \prod_{k=M_1}^{M_1+M_2-1} (1+r_k\eta)
\underset{\left \{ i,j \right \} \in K_1}{\max} |(\tilde{\mathbf Q}^{(M_1)}_{\mathcal I})_{ij}|.$$
Hence,
\begin{align}\label{eqs11}
|(\tilde{\mathbf Q}^{(M_1+M_2)}_{\mathcal I})_{i_mj_m}|
&\le \sum_{k=m+1}^{M_1+M_2-1} r_k\eta|(\tilde{\mathbf Q}^{(m)}
_{\mathcal I})_{i_kj_k}|
+|(\tilde{\mathbf Q}^{(m+1)}_{\mathcal I})_{i_mj_m}|
\nonumber\\&\le \sum_{k=M_1}^{M_1+M_2-1} r_k\eta|(\tilde{\mathbf Q}^{(m)}
_{\mathcal I})_{i_kj_k}|
\nonumber\\&\le \sum_{k=M_1}^{M_1+M_2-1} r_k\eta \max_{\substack{\left \{ i,j  \right \} \in K_1 \\ p\in K_3}} |(\tilde{\mathbf Q}^{(p)}_{\mathcal I})_{ij}|
\nonumber\\&\le \alpha \beta \eta\underset{\left \{ i,j \right \} \in K_1}{\max} |(\tilde{\mathbf Q}^{(M_1)}_{\mathcal I})_{ij}|.
\end{align}
Similarly, we have
\begin{equation}\label{eqs12}
\underset{ i}{\max} |
(\tilde{\mathbf Q}^{(M_1+M_2)}_{\mathcal I})_{ii}|
\le \alpha \beta \eta\underset{\left \{ i,j \right \} \in K_1}{\max}
|(\tilde{\mathbf Q}^{(M_1)}_{\mathcal I})_{ij}|
+\underset{i}{\max} |
(\tilde{\mathbf Q}^{(M_1)}_{\mathcal I})_{ii}|,
\end{equation}
and
\begin{equation}\label{eqs13}
\underset{\left \{ i,j \right \}\in K_2 }{\max} |
(\tilde{\mathbf Q}^{(M_1+M_2)}_{\mathcal I})_{ij}|
\le \alpha \beta \eta\underset{\left \{ i,j \right \} \in K_1}{\max}
|(\tilde{\mathbf Q}^{(M_1)}_{\mathcal I})_{ij}|
+\underset{\left \{ i,j \right \}\in K_2 }{\max} |
(\tilde{\mathbf Q}^{(M_1)}_{\mathcal I})_{ij}|.
\end{equation}
It follows from (\ref{eqs11}), (\ref{eqs12}), (\ref{eqs13}) and Corollary \ref{c3.2} that
\begin{equation}\label{eqs111}
\underset{\left \{ i,j \right \}\in K_1}{\max} |
(\tilde{\mathbf Q}^{(M_1+M_2)}_{\mathcal I})_{ij}|
\le\alpha \beta \eta\| \tilde{\mathbf Q}_{\mathcal I}\|_F,
\end{equation}
\begin{equation}\label{eqs112}\underset{ i}{\max} |(\tilde{\mathbf Q}^{(M_1+M_2)}_{\mathcal I})_{ii}|
\le(1+\alpha \beta \eta)\| \tilde{\mathbf Q}_{\mathcal I}\|_F,
\end{equation}
and
\begin{equation}\label{eqs113}
\underset{\left \{ i,j \right \}\in K_2 }{\max} |(\tilde{\mathbf Q}^{(M_1+M_2)}_{\mathcal I})_{ij}|
\le(1+\alpha \beta \eta)\| \tilde{\mathbf Q}_{\mathcal I}\|_F.
\end{equation}
According to Algorithm \ref{Jacobi_method_algorithm_3}, the number of iterations in STEP 2 is at most \begin{equation}\label{eqs1.20}T_2\doteq\frac{n^2}{2}.\end{equation}
Without loss of generality, suppose that $$\operatorname{diag}
(\tilde{\mathbf Q}^{(M_1+M_2)}_{st})=\operatorname{diag}(\lambda_1,\lambda_2,\cdots,\lambda_n ), $$ where
$$\lambda_1\approx \lambda_2\approx\cdots\approx\lambda_{t_1},$$ $$\lambda_{t_1+1}
\approx \lambda_{t_1+2}\approx\cdots\approx\lambda_{t_1+t_2},$$
$$\cdots$$
$$\lambda_{\sum_{i=1}^{p-1} t_i+1}\approx \lambda_{\sum_{i=1}^{p-1} t_i+2} \approx\cdots\approx\lambda_{\sum_{i=1}^{p}t_i}$$ and $\sum_{i=1}^{p}t_i=n$.
Based on STEP 2 in Algorithm \ref{Jacobi_method_algorithm_3}, for $1\le j_1<j_2\le t_i$, it holds $$|\lambda_{\sum_{k=1}^{i-1} t_k+j_1}-\lambda_{\sum_{k=1}^{i-1} t_k+j_2}|\le \gamma .$$
Denote $L=\hat{\mathbf J}_3$, then $L= \operatorname{diag}(L_1,L_2,\cdots,L_p) $, where $L_i \in \mathbb Q^{t_i\times t_i}$.
Suppose that $$\mu_i =\frac{1}{t_i} \sum_{j=I_1}^{I_2} \lambda_j,\quad I_1=\sum_{k=1}^{i-1} t_k+1,\quad I_2=\sum_{k=1}^{i} t_k,$$
denote $\mathbf Q_1=\operatorname{diag}(\mu_1 I_{t_1},\mu_2 I_{t_2},\cdots,\mu_p I_{t_p})$ and $ \mathbf Q_2=\tilde{\mathbf Q}_{st}^{(M_1+M_2)} -\mathbf Q_1$, then $|\mathbf Q_2|_{ii}\le \gamma $ and $|\mathbf Q_2|_{ij}\le \eta$ for $i\ne j$.
Thus, $$\tilde{\mathbf Q}_{st}^{(M)} =L^{\ast}\tilde{\mathbf Q}_{st}^{(M_1+M_2)}L
=L^{\ast}(\mathbf Q_1+\mathbf Q_2)L=\mathbf Q_1+L^{\ast}\mathbf Q_2L.$$ So,  for $i\ne j$, we have
$$|(\tilde{\mathbf Q}_{st}^{(M)})_{ij}|
\le|(L^{\ast}\mathbf Q_2L)_{ij}|\le \gamma +(h_1-1)\eta \le 3n\eta.$$
Due to (\ref{eqs113}) and similar to STEP 1, the number of iterations in STEP 3 is at most
\begin{equation}\label{eqt3}
T_3\doteq\frac{h_2(1+\alpha \beta \eta)^2\| \tilde{\mathbf Q}_{\mathcal I} \|_F^2}{2\delta_1^2} +\frac{h_2}{2\rho^2}\left \lceil \log_{\rho}{\frac{\eta}{\delta_1}}  \right \rceil,
\end{equation}
and $|\tilde{\mathbf Q}_{\mathcal I}^{(M)}|_{ij} \le \eta$ for $\left \{ i,j \right \}\in K_2$.

We rewrite $\tilde{\mathbf Q}_{\mathcal I}^{(M_1+M_2)}=(H^1_{ij})$ and $\tilde{\mathbf Q}_{\mathcal I}^{(M)}=(H^2_{ij})$, whose blocks are the same to $L$.
Since $\tilde{\mathbf Q}_{\mathcal I}^{(M)}=L^\ast \tilde{\mathbf Q}_{\mathcal I}^{(M_1+M_2)}L$, then $H^2_{ij}=L_i^\ast H^1_{ij} L_j$. Thus, for $\{ i,j \}\in K_1$, it holds
$$|(\tilde{\mathbf Q}_{\mathcal I}^{(M)})_{ij}| \le
h_1\underset{(i,j)\in K_1}{\max} |(\tilde{\mathbf Q}_{\mathcal I}^{(M_1+M_2)})_{ij}|,\quad
|(\tilde{\mathbf Q}_{\mathcal I}^{(M)})_{ii}| \le
h_1~\underset{ij}{\max}|(\tilde{\mathbf Q}_{\mathcal I}^{(M_1+M_2)})_{ij}|.$$


In summary, for $i\ne j$, we have $$|(\tilde{\mathbf Q}_{st}^{(M)})_{ij}| \le 3n\eta, \quad |(\tilde{\mathbf Q}_{\mathcal I}^{(M)})_{ij}| \le \max\left \{ \eta , h_1(\alpha \beta \eta) \| \tilde{\mathbf Q}_{\mathcal I}  \|_F \right \} .$$
For $i=j$, we have $$|(\tilde{\mathbf Q}_{\mathcal I}^{(M)})_{ii}|\le h_1(1+\alpha \beta \eta)\| \tilde{\mathbf Q}_{\mathcal I} \|_F=h_1\kappa\| \tilde{\mathbf Q}_{\mathcal I} \|_F.$$
Then, from (\ref{eqt1}), (\ref{eqs1.20}) and (\ref{eqt3}),  (\ref{wucha1}) and (\ref{wucha2}) hold after at most $T$ iterations, where
$$T=T_1+T_2+T_3=\frac{\| \tilde{\mathbf Q}_{st}\|_F^2}{2\delta^2} +
\frac{n^2}{2}\left ( 1+\frac{\left \lceil  \log_{\rho}{\frac{\eta}{\delta}}\right\rceil}{\rho^2} \right ) +
\frac{h_2}{2}\left ( \frac{\kappa^2
\| \tilde{\mathbf Q}_{\mathcal I} \|_F^2}{\delta_1^2}+\frac{\left \lceil \log_{\rho}{\frac{\eta}{\delta_1}}\right \rceil}{\rho^2}
 \right ).$$
\end{proof}

\begin{remark}\label{remark1} From (\ref{wucha1}) in Theorem \ref{theoreM_2}, it is easy to see that Algorithm \ref{Jacobi_method_algorithm_3} can output an $O(\eta)$-approximal diagonal dual number matrix $\hat{\mathbf Q}^{(M)}$ after at most $T$ iteration, whose diagonal elements are approximation of eigenvalues of $\hat{\mathbf Q}$.

Especially, STEP 2 can be repeated many times to accelerate the elimination process. Suppose that we repeat STEP 2 for $S$ times. Similar to the derivation processes of (\ref{eqs111}), (\ref{eqs112}) and (\ref{eqs113}), we can obtain
\begin{equation}\label{wuchas1}\underset{i\ne j}{\max}|(\tilde{\mathbf Q}_{st}^{(M)})_{ij}| \le 3n\eta,\quad \underset{i\ne j}{\max}|(\tilde{\mathbf Q}_{\mathcal I}^{(M)})_{ij}| \le \max\left \{ \eta ,h_1(\alpha \beta \eta)^S \| \tilde{\mathbf Q}_{\mathcal I}\|_F \right \},\end{equation} and
\begin{equation}\label{wuchas2}
\underset{i}{\max} |(\tilde{\mathbf Q}_{\mathcal I}^{(M)})_{ii}|\le h_1\kappa_S\| \tilde{\mathbf Q}_{\mathcal I}\|_F \end{equation}
after $T_S$ iterations, where
\begin{equation}\label{TS}
T_S=\frac{\| \tilde{\mathbf Q}_{st} \|_F^2}{2\delta^2} +
\frac{n^2}{2}\left ( S+\frac{\left \lceil  \log_{\rho}{\frac{\eta}{\delta}}\right\rceil}{\rho^2} \right ) +
\frac{h_2}{2}\left ( \frac{\kappa_S^2
\| \tilde{\mathbf Q}_{\mathcal I} \|_F^2}{\delta_1^2}+\frac{\left \lceil \log_{\rho}{\frac{\eta}{\delta_1}}\right \rceil}{\rho^2}
 \right )
 \end{equation}
 and
 $$
 \kappa_S=\frac{1-(\alpha\beta\eta)^{S+1}}{1-\alpha\beta\eta}.$$
\end{remark}

Under some mild assumptions, we can obtain more precise results than $O(\eta)$-approximation in Theorem \ref{theoreM_2}. We need the following assumption.
\begin{assumption}\label{A2}
Let $\hat{\mathbf Q}=\tilde{\mathbf Q}_{st} +\tilde{\mathbf Q}_{\mathcal I}\varepsilon\in\hat{ \mathbb{H}}^{n}$ be a dual quaternion Hermitian matrix.  We denote the eigenvalues of $\tilde{\mathbf Q}_{st}$ as $ \lambda _1>\lambda _2>\cdots > \lambda _p$ and the multiplicity of each eigenvalue $\lambda _i$ as $t_i$. Suppose that there exists a constant $c>0$ such that $\underset{i<j}{\min} (\lambda _i-\lambda _j)\ge c$.
\end{assumption}

\begin{corollary}\label{corollary3.7}
Let $\alpha,\beta$ be defined by (\ref{albe}). If $\hat{\mathbf Q}$ satisfies Assumption \ref{A2}, then
$$\alpha\le \frac{n(n-1)}{c-\gamma},\quad \beta\le(1+\frac{2\eta}{c-\gamma} )^{\frac{n(n-1)}{2}}.$$
Furthermore, suppose that $c=12n^2\eta$ in Assumption \ref{A2}, $\delta=\delta_1=1$ and $\rho^2=0.1$ in Algorithm \ref{Jacobi_method_algorithm_3}, STEP 2 is repeated $S$ times, where
$$S=\Big\lceil \log_{10}\frac{\| \tilde{\mathbf Q}_{\mathcal I}\|_F}{n\eta}\Big\rceil,$$
then $\alpha \beta \eta\le \frac{1}{10}$ and Algorithm \ref{Jacobi_method_algorithm_3} terminates after at most $T$ iterations with
$$\underset{i\ne j}{ \max}|(\tilde{\mathbf Q}_{st}^{(M)})_{ij}| \le 3n\eta,\quad \underset{i\ne j}{\max}|(\tilde{\mathbf Q}_{\mathcal I}^{(M)})_{ij}| \le  h_1n\eta,$$
and
$$\underset{ i}{\max} |(\tilde{\mathbf Q}^{(M)}_{\mathcal I})_{ii}|
\le 2h_1\| \tilde{\mathbf Q}_{\mathcal I} \|_F,$$
where $$T=\frac{1}{2}\| \tilde{\mathbf Q}_{st} \|_F^2+h_2\| \tilde{\mathbf Q}_{\mathcal I} \|_F^2+\frac{n^2}{2}\Big \lceil \log_{10}\frac{\| \tilde{\mathbf Q}_{\mathcal I} \|_F}{n\eta}\Big \rceil+10(n^2+h_2)\Big \lceil\log_{10}{\frac{1}{\eta}}\Big\rceil.$$
\end{corollary}

\begin{proof}
Without loss of generality, suppose that $$(\tilde{\mathbf Q}^{(M_1)}_{st})_{11}\ge(\tilde{\mathbf Q}^{(M_1)}_{st})_{22}\ge\cdots\ge(\tilde{\mathbf Q}^{(M_1)}_{st})_{nn}$$ and
$$\mu_1\ge\mu_2\ge\cdots\ge\mu_n$$ are eigenvalues of $\tilde{\mathbf Q}_{st}$.
Then, from (\ref{H-W}), we get
$$\sum_{i=1}^{n} (\mu_i-(\tilde{\mathbf Q}^{(M_1)}_{st})_{ii})^2\le n(n-1) \underset{i\ne j}{\max}| (\tilde{\mathbf Q}^{(M_1)}_{st})_{ij}| ^2\le n(n-1)\eta^2,$$
which, together with Assumption \ref{A2}, implies
\begin{align*}
|(\tilde{\mathbf Q}^{(M_1)}_{st})_{i_mi_m}-
(\tilde{\mathbf Q}^{(M_1)}_{st})_{j_mj_m}|
&\ge|\mu_{i_m}-\mu_{j_m}|-|(\tilde{\mathbf Q}^{(M_1)}_{st})_{i_mi_m}-
\mu_{i_m}|
\\&\qquad-|(\tilde{\mathbf Q}^{(M_1)}_{st})_{j_mj_m}-\mu_{j_m}|
\\&\ge c-\gamma .
\end{align*}
Hence, we get
$$r_m=\frac{2}{| (\tilde{\mathbf Q}^{(M_1)}_{st})_{i_mi_m}-
(\tilde{\mathbf Q}^{(M_1)}_{st})_{j_mj_m}|}\le \frac{2}{c-\gamma },$$
which yields $$\alpha=\sum_{k=M_1}^{M_1+M_2-1} r_k\le \frac{n(n-1)}{c-\gamma}$$ and $$\beta=\prod_{k=M_1}^{M_1+M_2-1} (1+r_k\eta)\le(1+\frac{2\eta}{c-\gamma} )^{\frac{n(n-1)}{2}}.$$

Furthermore, it follows from  $c=12n^2\eta$ that
\begin{align*}
\alpha\beta\eta &\le \frac{n(n-1)}{c-\gamma}(1+\frac{2\eta}{c-\gamma} )^{\frac{n(n-1)}{2}}\eta
\\&=\frac{n(n-1)\eta}{12n^2\eta-\sqrt{2n(n-1)}\eta}(1+\frac{2\eta}{12n^2\eta-\sqrt{2n(n-1)}\eta} )^{\frac{n(n-1)}{2}}
\\&\le \frac{n^2}{12n^2-2n}(1+\frac{2}{12n^2-2n})^{\frac{n(n-1)}{2}}
\\&\le \frac{1}{11}e^{\frac{n(n-1)}{12n^2-2n}}\le \frac{1}{11}e^{\frac{1}{11}}\le \frac{1}{10}.
\end{align*}

Since it sets $\delta=\delta_1=1$ and $\rho^2=0.1$ in Algorithm \ref{Jacobi_method_algorithm_3} and STEP 2 is repeated $S$ times, by (\ref{wuchas1}) and (\ref{wuchas2}) , we have
\begin{align*}
\underset{i\ne j}{\max}|(\tilde{\mathbf Q}_{\mathcal I}^{(M)})_{ij}|
&\le\max\left \{ \eta, h_1(\alpha \beta \eta)^S \| \tilde{\mathbf Q}_{\mathcal I} \|_F \right \}
\le \max\left \{ \eta, \frac{h_1 \| \tilde{\mathbf Q}_{\mathcal I}\|_F }{10^S} \right \}
\\&\le \max\left \{ \eta, h_1n\eta\right \}
= h_1n\eta,
\end{align*}
and
\begin{align*}
\underset{i}{\max}|(\tilde{\mathbf Q}_{\mathcal I}^{(M)})_{ii}|
&\le h_1\kappa_S \| \tilde{\mathbf Q}_{\mathcal I}  \|_F
\le h_1\frac{1}{1-\alpha \beta \eta} \| \tilde{\mathbf Q}_{\mathcal I} \|_F
\\&\le h_1\frac{1}{1-\frac{1}{10}} \|
\tilde{\mathbf Q}_{\mathcal I} \|_F
\le 2h_1\| \tilde{\mathbf Q}_{\mathcal I} \|_F.
\end{align*}

It follows from $\delta=\delta_1=1$ and $\rho^2=0.1$ that
$$\frac{\kappa_S^2}{2\delta_1^2}\le \frac{1}{2(1-\frac{1}{10})^2}\le 1,\quad
\frac{\Big \lceil  \log_{\rho}{\frac{\eta}{\delta}}\Big\rceil}{\rho^2} =
\frac{\Big \lceil  \log_{\rho}{\frac{\eta}{\delta_1}}\Big\rceil}{\rho^2}
=20\Big \lceil \log_{10}{\frac{1}{\eta}}\Big\rceil.$$
Hence, by (\ref{TS}),  the number of iterations is at most $$T=\frac{1}{2}\| \tilde{\mathbf Q}_{st} \|_F^2+h_2 \| \tilde{\mathbf Q}_{\mathcal I} \|_F^2+\frac{n^2}{2}\Big\lceil \log_{10}\frac{\| \tilde{\mathbf Q}_{\mathcal I}\|_F}{n\eta}\Big \rceil+10(n^2+h_2)\Big \lceil\log_{10}{\frac{1}{\eta}}\Big\rceil.$$
\end{proof}

We now discuss the case that the dual quaternion Hermitian matrix $\hat{\mathbf Q}$ has $n$ simple eigenvalues. In this case, by {\it Remark} \ref{remark1} and Corollary \ref{corollary3.7}, we  immediately obtain the following results.
\begin{corollary}\label{corollary3.8}
If $\hat{\mathbf Q}$ has $n$ simple eigenvalues,
Algorithm \ref{Jacobi_method_algorithm_3} terminates after at most
$$\frac{\| \tilde{\mathbf Q}_{st}\|_F^2}{2\delta^2}
+\frac{n^2}{2\rho^2}\left \lceil \log_{\rho}{\frac{\eta}{\delta}}\right\rceil+\frac{n^2S}{2}$$ iterations with
$$\underset{i\ne j}{ \max}|(\tilde{\mathbf Q}_{st}^{(M)})_{ij}| \le \eta,\quad \underset{i\ne j}{\max}|(\tilde{\mathbf Q}_{\mathcal I}^{(M)})_{ij}| \le \max\Big \{ \eta ,(\alpha \beta \eta)^S \| \tilde{\mathbf Q}_{\mathcal I}\|_F \Big \},$$
and
$$\underset{i}{\max} |(\tilde{\mathbf Q}_{\mathcal I}^{(M)})_{ii}|\le\kappa_S\| \tilde{\mathbf Q}_{\mathcal I}\|_F.$$

If $\hat{\mathbf Q}$ also satisfies the conditions given in Corollary \ref{corollary3.7}, then Algorithm \ref{Jacobi_method_algorithm_3} terminates after at most
\begin{equation}\label{eqT}T\doteq\frac{1}{2}\|\tilde{\mathbf Q}_{st}\|_F^2 +10n^2\Big\lceil \log_{10}{\frac{1}{\eta}}\Big\rceil+\frac{n^2}{2}\Big\lceil \log_{10}\frac{\| \tilde{\mathbf Q}_{\mathcal I} \|_F}{n\eta}\Big \rceil
\end{equation}
iterations with \begin{equation}\label{wucha}\underset{i\ne j}{ \max}|(\tilde{\mathbf Q}_{st}^{(M)})_{ij}| \le \eta,\quad \underset{i\ne j}{\max}|(\tilde{\mathbf Q}_{\mathcal I}^{(M)})_{ij}| \le n\eta,\quad \underset{ i}{\max} |(\tilde{\mathbf Q}^{(M)}_{\mathcal I})_{ii}|
\le 2\| \tilde{\mathbf Q}_{\mathcal I}\|_F.\end{equation}
\end{corollary}

The following theorem shows that if the dual quaternion Hermitian matrix $\hat{\mathbf Q}$ has $n$ simple eigenvalues then Algorithm \ref{Jacobi_method_algorithm_3} can provide their $\eta$-approximal solution after at most $T_S$ iterations.
\begin{theorem}\label{theorem3.9}
Suppose that $\hat{\mathbf Q}$ has $n$ simple eigenvalues $\hat{\lambda}_i=\lambda^s_i+\lambda^d_i\varepsilon\ (i=1,\ldots,n)$ and the conditions in Corollary \ref{corollary3.8} holds. Let $T$ be defined as  (\ref{eqT}) and $\eta>0$ be given accuracy parameter in Algorithm \ref{Jacobi_method_algorithm_3}. If $\underset{i\ne j}{\min}|\lambda^s_i-\lambda^s_j|\ge12n^2\eta^{1-\nu}$ where $\nu\in(0,1)$, then
Algorithm \ref{Jacobi_method_algorithm_3} can make sure that
$$
\max_{1\le i\le n}|\lambda^s_i- (\tilde{\mathbf Q}^{(M)}_{st})_{ii}|\le \sqrt{n(n-1)}\eta,\quad \max_{1\le i\le n}|\lambda^d_i- (\tilde{\mathbf Q}^{(M)}_{\mathcal I})_{ii}|\le \eta+\frac{\| \tilde{\mathbf Q}_{\mathcal I}\|_F}{n^2}\eta^\nu
$$
after at most $T$ iterations. That is, Algorithm \ref{Jacobi_method_algorithm_3} can provide $\varepsilon$-approximation with $\epsilon\doteq\eta^{\nu}$ for each eigenvalue $\hat{\lambda}_i\ (i=1,\ldots,n)$ in at most $T$ iterations.
\end{theorem}

\begin{proof}
Without loss of generality, suppose that $\lambda^s_1\ge\lambda^s_2\ge\cdots\ge\lambda^s_n$ and $$(\tilde{\mathbf Q}^{(M_1)}_{st})_{11}\ge(\tilde{\mathbf Q}^{(M_1)}_{st})_{22}\ge\cdots\ge(\tilde{\mathbf Q}^{(M_1)}_{st})_{nn}.$$
Since $\underset{i\ne j}{\min}|\lambda^s_i-\lambda^s_j|\ge12n^2\eta^{1-\nu}\ge 12n^2\eta$, by Corollary \ref{corollary3.8}, (\ref{wucha}) holds.   From the first inequality in (\ref{wucha}) and (\ref{H-W}), we get
$$\sum_{i=1}^{n} (\lambda^s_i-(\tilde{\mathbf Q}^{(M)}_{st})_{ii})^2\le n(n-1) \underset{i\ne j}{\max} | (\tilde{\mathbf Q}^{(M)}_{st})_{ij}  | ^2\le n(n-1)\eta^2,$$
which implies
\begin{equation}\label{chas}
\max_{1\le i\le n}|\lambda^s_i- (\tilde{\mathbf Q}^{(M)}_{st})_{ii}|\le \sqrt{n(n-1)}\eta=\big(\sqrt{n(n-1)}\eta^{1-\nu}\big)\epsilon.
\end{equation}

Since $\underset{i\ne j}{\min}|\lambda^s_i-\lambda^s_j|\ge12n^2\eta^{1-\nu}$, we have $c=12n^2\eta^{1-\nu}\ge 2\sqrt{n(n-1)}\eta$. Hence, by Lemma \ref{lemma5}, there exist a unitary quaternion matrix $\tilde{\mathbf V}$ which diagonalises $\tilde{\mathbf Q}_{st}^{\left ( M\right )}$, and
$$\underset{ij}{\max}|(\tilde{\mathbf{I}}_n - \tilde{\mathbf V})_{ij}|
\le \frac{\eta^\nu}{6n^2}.$$

Since $\tilde{\mathbf V}^*\hat{\mathbf Q}^{(M)}\tilde{\mathbf V}=\tilde{\mathbf V}^*\tilde{\mathbf Q}^{(M)}_{st}\tilde{\mathbf V}+\tilde{\mathbf V}^*\tilde{\mathbf Q}^{(M)}_{\mathcal I}\tilde{\mathbf V}$, where $\tilde{\mathbf V}^*\tilde{\mathbf Q}^{(M)}_{st}\tilde{\mathbf V}$ is a diagonal matrix, according to Lemma \ref{lemma4}, the diagonal elements of
$\tilde{\mathbf V}^*\hat{\mathbf Q}^{(M)}\tilde{\mathbf V}$ are the eigenvalues of $\hat{\mathbf Q}^{(M)}$, then $\lambda^d _i=\mathbf v_i^*\tilde{\mathbf Q}^{(M)}_{\mathcal I}\mathbf v_i$.
Hence, for $i=1,\ldots,n$ we have
\begin{align*}
|  \lambda^d _i- (\tilde{\mathbf Q}^{(M)}_{\mathcal I})_{ii} |
&= |\mathbf v_i^*\tilde{\mathbf Q}^{(M)}_{\mathcal I}\mathbf v_i
-(\tilde{\mathbf Q}^{(M)}_{\mathcal I})_{ii} |
\\&\le |(\mathbf v_i^*-e_i^T)\tilde{\mathbf Q}^{(M)}
_{\mathcal I}(\mathbf v_i-e_i)|+2 |e_i^T\tilde{\mathbf Q}^{(M)}
_{\mathcal I}(\mathbf v_i-e_i) |
\\&\le \frac{\eta^{2\nu}}{36n^4}\sum_{ij}^{}  |(\tilde{\mathbf Q}^{(M)}_{\mathcal I}) _{ij} |+\frac{\eta^\nu}{3n^2}\underset{i}{\max}\sum_{j}^{}  |(\tilde{\mathbf Q}^{(M)}_{\mathcal I}) _{ij}   |
\\&\le \frac{\eta^{2\nu}}{36n^4}(n^3\eta+2n \| \tilde{\mathbf Q}_{\mathcal I} \|_F )+\frac{\eta^\nu}{3n^2}(n^2\eta+2 \| \tilde{\mathbf Q}_{\mathcal I}  \|_F )
\\&=(\frac{\eta^{2\nu}}{36n^3}+\frac{\eta^\nu}{3n^2})(n^2\eta+2 \| \tilde{\mathbf Q}_{\mathcal I}  \|_F )
\\&\le \eta+\frac{ \| \tilde{\mathbf Q}_{\mathcal I} \|_F}{n^2}\eta^\nu,
\end{align*}
which implies
\begin{equation}\label{chad}
\max_{1\le i\le n}|\lambda^d_i- (\tilde{\mathbf Q}^{(M)}_{\mathcal I})_{ii}|\le \eta+\frac{\| \tilde{\mathbf Q}_{\mathcal I}\|_F}{n^2}\eta^\nu=\Big(\eta^{1-\nu}+\frac{\| \tilde{\mathbf Q}_{\mathcal I}\|_F}{n^2}\Big)\epsilon.
\end{equation}
Clearly, it shows from (\ref{chas}) and (\ref{chad}) that
Algorithm \ref{Jacobi_method_algorithm_3} can provide $\epsilon$-approximation for each eigenvalue $\hat{\lambda}_i\ (i=1,\ldots,n)$ in at most $T$ iterations.
\end{proof}

\section{Numerical Experiments}\label{Experiments}

In this section, we apply Algorithm \ref{Jacobi_method_algorithm1}, Algorithm \ref{Jacobi_method_algorithm2}, Algorithm \ref{Jacobi_method_algorithm_3} to compute eigenvalues for a given dual quaternion Hermitian matrix $\hat{\mathbf Q}\in \hat{\mathbb H}^{n}$. We compare the performance of Algorithm \ref{Jacobi_method_algorithm_3} with that of the power method \cite{qc4} and Rayleigh quotient iteration (RQI) method \cite{qc17} to show the efficiency of our algorithm.

Suppose that $\{ \hat{\mathbf Q}^{(t)}= \tilde{\mathbf Q}_{st}^{(t)}+\tilde{\mathbf Q}^{(t)}_{\mathcal I}\varepsilon\}$ is the generated sequence. Let $D=\| \hat{\mathbf Q}\|_{F^R}$. We use
$$R^t\doteq\frac{1}{D} \left (\sum_{i\ne j}^{}(| (\tilde{\mathbf Q}^{(t)}_{st})_{ij}
|^2 +  | (\tilde{\mathbf Q}^{(t)}_{\mathcal I})_{ij} |^2 )   \right ) ^\frac{1}{2}$$ to
verify whether the non-diagonal elements tend to zero. Denote
\begin{equation*}
    e_{\lambda }^t\doteq\frac{1}{n}\sum_{i=1}^{n}\| \hat{\mathbf Q}\hat{\mathbf u}_{i}^{(t)}-\hat{\lambda }_{i}^{(t)}\hat{\mathbf u}_{i}^{(t)} \|_{2^{R}},
\end{equation*}
where $\hat{\lambda }_{i}^{(t)}$ and $\hat{\mathbf u}_{i}^{(t)}$  are eigenvalues and eigenvectors generated by these algorithms at $t$-th iteration. We use $e_{\lambda }^t$ to verify the accuracy of the output of these algorithms.  All numerical experiments are conducted in MATLAB (2022a) on a laptop of 8G of memory and Inter Core i5 2.3Ghz CPU. For Algorithm \ref{Jacobi_method_algorithm1}
we set accuracy parameter $\epsilon=1e-7$. For Algorithm \ref{Jacobi_method_algorithm2} we set parameter $\delta=1$, $\eta=1e-7$ and $\rho^2=0.1$. For Algorithm \ref{Jacobi_method_algorithm_3} we set parameter $\delta=\delta_1=1$, $\eta=1e-7$, $\rho^2=0.1$ and we repeat STEP 2 for 2 times, i.e., we set $S=2$.

\subsection{Performance of Algorithm \ref{Jacobi_method_algorithm1} and Algorithm \ref{Jacobi_method_algorithm2}}
We first show the sequence generated by Algorithm \ref{Jacobi_method_algorithm1} and Algorithm \ref{Jacobi_method_algorithm2} converges to a diagonal dual number matrix. The numerical performance is  consistent with the theoretical results obtained in Theorems \ref{thmst} and \ref{thmd}.

We randomly generate a dual quaternion Hermitian matrix $\hat{\mathbf Q}$ with dimension $n=30$, and use Algorithm \ref{Jacobi_method_algorithm1} and Algorithm \ref{Jacobi_method_algorithm2} to compute the eigenvalues and eigenvectors. Denote $\{ R^t_1\}_{t=1}^{N_1}$, $\{ e_{\lambda ,1}^t \}_{t=1}^{N_1} $ generated by Algorithm \ref{Jacobi_method_algorithm1}  and $\{ R^t_2 \}_{t=1}^{N_2}$, $\{ e_{\lambda,2}^t \}_{t=1}^{N_2}$ generated by Algorithm \ref{Jacobi_method_algorithm2}.  After $N_1=1778$ and $N_2=1966$ iterations respectively, Algorithm \ref{Jacobi_method_algorithm1} and Algorithm \ref{Jacobi_method_algorithm2} terminate and output
$$e_{\lambda ,1}^{N_1}=3.65\times 10^{-6},\quad e_{\lambda ,2}^{N_2}=2.58\times 10^{-6},$$ $$R^{N_1}_1=6.31\times 10^{-6},\quad R^{N_2}_2=5.31\times 10^{-7}.$$ This shows that both Algorithm \ref{Jacobi_method_algorithm1} and Algorithm \ref{Jacobi_method_algorithm2} can accurately compute eigenvalues and eigenvectors for dual quaternion Hermitian matrices, and provide a diagonal dual number matrix after finite iterations (see Figure \ref{fig1}). The elapsed CPU time (sec.) for Algorithm \ref{Jacobi_method_algorithm1} and Algorithm \ref{Jacobi_method_algorithm2} is $0.8555s$ and $0.3450s$, respectively. This indicates that the accelerative strategy  used in Algorithm \ref{Jacobi_method_algorithm2} can help to reduce the computing time.
\begin{figure}[h]
    \centering
    \includegraphics[width=0.6\linewidth]{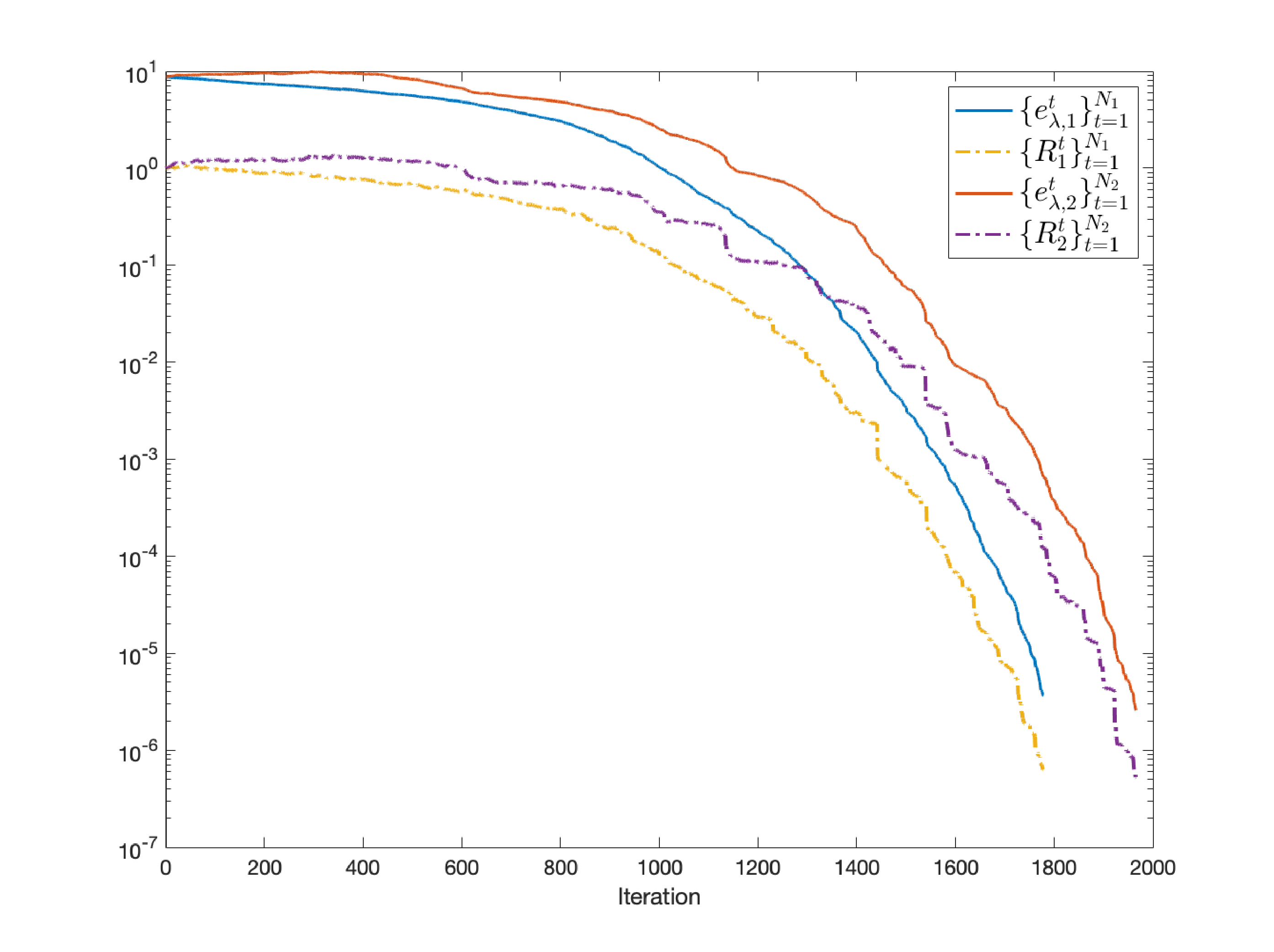} \vspace{-3mm}
    \caption{\small Numerical performance of Algorithm \ref{Jacobi_method_algorithm1} and Algorithm \ref{Jacobi_method_algorithm2}.}\label{fig1}
    \end{figure}

\subsection{Performance of  Algorithm \ref{Jacobi_method_algorithm_3}} We next to show the performance of Algorithm \ref{Jacobi_method_algorithm_3}. We generate a dual quaternion Hermitian matrix $\hat{\mathbf{P}}=(\hat{p}_{ij})$ as the following way. Let $\hat{\mathbf{q}}=(\hat{q}_i)$ be a random unit dual quaternion vector defined by
$$\small{
\setlength{\arraycolsep}{1.2pt}
\left.\hat{\mathbf{q}}=\left[\begin{array}{rrrr}0.9359&0.3033&0.0112&-0.1785\\-0.6476&0.3307&0.6751&-0.1249\\-0.7964&-0.4063&0.4446&0.0542\\-0.4627&-0.3857&-0.7755&-0.1891\\-0.4083&-0.4844&-0.7025&-0.3243\end{array}\right.\right]+
\left[\begin{array}{rrrr}0.0739&-0.9213&-1.0193&-1.2419\\-0.2448&-0.0200&-0.3720&-0.7944\\
   -0.3142&0.0313&-0.5714&0.3056\\0.2159&-0.5179&0.1159&0.0530\\
   -0.1260&0.1389&0.0662&-0.1923\end{array}\right]\varepsilon}.$$
Let $E=\left \{  \left \{ 1,2 \right \},\left \{ 2,3 \right \},\left \{ 3,4\right \},\left \{ 4,5 \right \},\left \{ 5,1 \right \}\right \}$ and set $\hat{p}_{ij}$ as
\begin{equation*}
\hat{p}_{ij}=
\begin{cases}
 \hat{q}_{i}^{\ast } \hat{q}_{j}, & \text{if}~ ~\left \{ i,j \right \} \in E ,\\
i\varepsilon , & \text{if}~ ~ i=j ,\\
0, & \text{otherwise.} \\
\end{cases}
\end{equation*}

We compute all eigenvalues of $\hat{\mathbf{P}}$ by Algorithm \ref{Jacobi_method_algorithm_3}. Algorithm \ref{Jacobi_method_algorithm_3} terminates with $e_{\lambda }^N=1.5341\times 10^{-8}$ and $R^N=3.1167\times 10^{-9}$, the elapsed CPU time is $0.0092s$ and all five eigenvalues of $\hat{\mathbf{P}}$ are found
as follows:
{\footnotesize
\begin{equation}\label{eig1}
2.0000+3.0000\varepsilon,\ 0.6180+3.5257\varepsilon,\ 0.6180+2.4743\varepsilon,\ -1.6180+3.8507\varepsilon,\ -1.6180+2.1493\varepsilon.
\end{equation}}
This shows that  Algorithm \ref{Jacobi_method_algorithm_3} computes all eigenvalues of $\hat{\mathbf{P}}$ accurately and quickly. By (\ref{eig1}), it is easy to see that $\hat{\mathbf{P}}$ has two eigenvalues with the same standard parts and different dual parts.  Algorithm \ref{Jacobi_method_algorithm_3} can compute it accurately and quickly, but both the power method \cite{qc4} and RQI method \cite{qc17} are no longer valid and cannot compute them. This shows the advantages of Algorithm \ref{Jacobi_method_algorithm_3}.

\subsection{On randomly generated dual quaternion Hermitian matrices}

We consider to compute the eigenvalues of random dual quaternion Hermitian matrices by Algorithm \ref{Jacobi_method_algorithm_3}. We randomly generate dual quaternion Hermitian matrices with dimension $n=10,50,100,150,200$ and repeat 50 times. All numerical results are summarized in Table \ref{table1}, in which $n_{iter}$ is the average number of iterations, {CPU time} $(s)$ is the average elapsed CPU time in seconds, $\sigma_T$ is the standard deviation of the elapsed CPU time and $\sigma_N$ is the standard deviation of the number of iterations.
\begin{table}[!h]
\renewcommand{\arraystretch}{1.2}

    \centering
    \caption{\small{Numerical results of Algorithm \ref{Jacobi_method_algorithm_3} for computing all eigenvalues of randomly generated  dual quaternion Hermitian matrices with different dimension}}\vspace{-0.2cm}\label{table1}
    \resizebox{9.5cm}{1.3cm}{
    \begin{tabular}{|c|c|c|c|c|c|c|}
        \hline
        $n$ & $e_{\lambda }^N$ & $R^N$  & CPU time $(s)$ & $n_{iter}$ & $\sigma_T$ & $\sigma_N$ \\
        \cline{1-7}
        10 & 8.81e-8 & 1.30e-8 & 1.45e-2  & 2.70e+2 & 4.82e-3 & 4.19 \\
        \cline{1-7}
        50 & 7.03e-7 &  2.59e-8 & 4.34e-1  & 8.12e+3 & 6.10e-2 & 30.79\\
        \cline{1-7}
        100 & 2.36e-6 & 4.00e-8 & 2.09 & 3.36e+4 & 3.73e-1 & 73.48 \\
        \cline{1-7}
        150 & 5.21e-6 & 6.62e-8 & 6.04 & 7.68e+4 & 4.98e-1 &  124.45 \\
        \cline{1-7}
        200 & 9.18e-6 & 8.87e-8 & 11.38 & 1.38e+5 & 7.18e-1 & 170.53\\
        \cline{1-7}
    \end{tabular}}
\end{table}

From Table \ref{table1}, we see that average elapsed CPU time for computing all eigenvalues of randomly generated dual quaternion Hermitian matrices with dimension $n=100$ is about $2s$. When $n=200$, the average elapsed CPU time for computing all eigenvalues is $11.38s$ and the error $e_{\lambda}^N$ is $9.18\times 10^{-6}$. These numerical results show that Algorithm \ref{Jacobi_method_algorithm_3} is effective for compute all eigenvalues of dual quaternion Hermitian matrices. Moreover, from the last two columns of Table \ref{table1}, we clearly find that the running time and the number of iterations for Algorithm \ref{Jacobi_method_algorithm_3} remain stable across different dual quaternion Hermitian matrices.

\subsection{On Laplacian matrices of randomly generated graphs} In this subsection, we use Algorithm \ref{Jacobi_method_algorithm_3} to compute all eigenvalues for Laplacian matrices of  randomly generated graphs.

Given a graph $G=(V, E)$ with $n$ vertices, its Laplacian matrix $\hat{\mathbf L}$ is
defined as
\begin{equation*}
    \hat{\mathbf L}=\hat{\mathbf D}-\hat{\mathbf A},
\end{equation*}
where $\hat{\mathbf D}$ is a diagonal real matrix whose diagonal elements equal to the
degrees of the corresponding vertices, and $\hat{\mathbf A}=\left ( \hat{a}_{ij}\right )$ is given as
\begin{equation*}
     \hat{ a}_{ij}=
    \begin{cases}
    \hat{ q}_{i}^{\ast } \hat{ q}_{j}, & \text{if} ~ \left ( i,j\right )\in E ,\\
    0, & \text{otherwise},
\end{cases}
\end{equation*}
where $\hat{\mathbf q}\in \hat{\mathbb{U}}^{n\times 1}$ is known in advance.

Assume that $G$ is a sparse graph and $E$ is symmetric with sparsity $s=|E|/n^{2}$, where $|E|$ is the number of elements in $E$. In practice, we randomly generate a graph with $s/2 \times n^{2}$ edges. We use the power method \cite{qc4}, RQI  method \cite{qc17}, and Algorithm \ref{Jacobi_method_algorithm_3} (3SJacobi) to compute all eigenvalues and eigenvectors of $\hat{\mathbf L}$ with $n=10$ or $100$ and different values of $s$. All results are repeated ten times with different choices of $\hat{\mathbf q}$ and different $E$ and the average numerical results are reported in Table \ref{table2}

\begin{table}[!h]
\renewcommand{\arraystretch}{1.2}
    \centering
    \caption{\small{ Numerical results of power method, RQI method and 3SJacobi for computing all eigenvalues of Laplacian matrices of random graphs with different sparsity}}\vspace{-0.2cm}\label{table2}
    \resizebox{11cm}{3cm}{
    \begin{tabular}{|c|c|c||c|c|c||c|c|c|}
        \hline
        \multicolumn{3}{|c||}{Power method}& \multicolumn{3}{c||}{RQI  method} & \multicolumn{3}{c|}{Algorithm \ref{Jacobi_method_algorithm_3}} \\
        \cline{1-9}
        \multicolumn{9}{|c|}{n=10}\\
        \cline{1-9}
        s & $e_{\lambda }^N$ & CPU time $(s)$ & s & $e_{\lambda }^N$ & CPU time $(s)$  &  s & $e_{\lambda }^N$ & CPU time $(s)$ \\
        \cline{1-9}
        0.1 & 1.34e-10 &  1.09e-1 &  0.1 & 2.59e-7 &  7.72e-2 &  0.1 & 1.06e-8 & 1.88e-2 \\
        \cline{1-9}
        0.2 &  4.99e-10 &  5.90e-1 &  0.2 & 5.86e-7 & 8.10e-2  &  0.2 & 3.27e-8 & 2.07e-2 \\
        \cline{1-9}
        0.3 & 6.94e-10 &  7.20e-1 & 0.3 & 3.89e-7 & 8.86e-2 & 0.3 & 4.07e-8 & 1.86e-2  \\
        \cline{1-9}
        0.4 & 9.69e-10 & 8.19e-1 & 0.4 & 5.22e-7 &  9.10e-2 &  0.4 & 5.97e-8 & 1.82e-2 \\
        \cline{1-9}
        0.5 &  1.21e-9 & 1.08 & 0.5 & 3.61e-7 &  9.11e-2 &  0.5 & 5.90e-8 & 2.16e-2 \\
        \cline{1-9}
        0.6 &  1.57e-9 &  1.37  & 0.6 &  3.37e-7 &  9.45e-2 & 0.6 &  5.05e-8 & 2.26e-2 \\
        \cline{1-9}
        \multicolumn{9}{|c|}{n=100}\\
        \cline{1-9}
        0.05 & 3.48e-6 &  46.21 & 0.05 & 6.00e-6 &  7.74 & 0.05 & 2.74e-7 & 1.66 \\
        \cline{1-9}
        0.08 &  3.09e-5 &  62.18 & 0.08 & 5.74e-6 &  8.16 &  0.08 & 3.26e-7 & 1.69 \\
        \cline{1-9}
        0.10 & 8.19e-5 &  72.13 & 0.10 & 4.72e-6 &  8.33 &  0.10 & 3.97e-7 & 1.73 \\
        \cline{1-9}
        0.15 & 7.46e-5 &  87.18 & 0.15 & 4.37e-6 &  8.20 &  0.15 & 5.96e-7 &  1.74 \\
        \cline{1-9}
        0.18 & 2.61e-4 & 101.4 & 0.18 & 3.97e-6 &  8.42 &  0.18 & 7.38e-7 & 1.82 \\
        \cline{1-9}
        0.2 &  1.84e-4 & 111.6  & 0.2 &  3.16e-6 & 9.12  &  0.2 &  8.64e-7 & 1.80 \\
        \cline{1-9}
    \end{tabular}}
\end{table}

Since RQI method requires a proper initial iteration point, we propose a further improvement on the basis of RQI method introduced in \cite{qc17}.
Prior to initiating the iteration process of RQI method, we incorporate a pre-processing step wherein a certain number of iterations of the power method are performed. This pre-processing step facilitates the acquisition of a more reasonable initial iteration point, thereby improving the efficiency of the Rayleigh quotient iteration method in practical.

Table \ref{table2} shows that Algorithm \ref{Jacobi_method_algorithm_3} (3SJacobi) outperforms power method \cite{qc4} and RQI  method \cite{qc17}. Algorithm \ref{Jacobi_method_algorithm_3} is very efficient for computing all eigenvalues of dual quaternion matrices in large scale. Especially, it has absolute advantage on the computation speed comparison basis.

\section{Final Remarks}\label{summery}

We propose three Jacobi-type methods to compute eigenvalues of dual quaternion Hermitian matrices and provide convergence analyses. Algorithm \ref{Jacobi_method_algorithm1} is a generalization of Jacobi eigenvalue method, which is shown that it can provide $\epsilon$-approximation of eigenvalues in finite iterations (see Theorems \ref{thmst} and \ref{thmd}). Algorithm \ref{Jacobi_method_algorithm2} improves Algorithm \ref{Jacobi_method_algorithm1} via adopting a new elimination strategy to reduce the computation cost (see Figure \ref{fig1}).  Since the existing method such as power method \cite{qc4} and Rayleigh quotient iteration method \cite{qc17} fail to  tackle the situation where a dual quaternion Hermitian matrix has two eigenvalues with identical standard parts but different dual parts, we propose a three-step Jacobi eigenvalue algorithm (Algorithm \ref{Jacobi_method_algorithm_3}), which effectively overcomes this question and compensates for the limitations of the existing methods (see(\ref{eig1})). We also show that Algorithm \ref{Jacobi_method_algorithm_3} can output an $O(\eta)$-approximal diagonal dual number matrix $\hat{\mathbf Q}^{(M)}$ after at most finite iterations, whose diagonal elements are the desired eigenvalues (see Theorem \ref{theoreM_2}). Moreover, it is shown that Algorithm \ref{Jacobi_method_algorithm_3} can provide $\epsilon$-approximation of  eigenvalues in finite iterations (see Theorem \ref{theorem3.9}) under some mild assumptions. Numerical results indicate that our 3SJacobi method (Algorithm \ref{Jacobi_method_algorithm_3}) outperforms the existing method such as power method \cite{qc4} and Rayleigh quotient iteration method \cite{qc17} (see Table \ref{table2}). Especially,  our 3SJacobi method has absolute advantage on the computation speed and can effectively compute all eigenvalues and eigenvectors for dual quaternion Hermitian matrices.

\bibliographystyle{amsplain}

\end{document}